\documentclass[reqno, 12pt]{amsart}
\pdfoutput=1
\makeatletter
\let\origsection=\section 
\def\section{\@ifstar{\origsection*}{\mysection}} 
\def\mysection{\@startsection{section}{1}\z@{.7\linespacing\@plus\linespacing}{.5\linespacing}{\normalfont\scshape\centering\S}}
\makeatother

\usepackage{amsmath,amssymb,amsthm}
\usepackage{mathrsfs}
\usepackage{mathabx}\changenotsign

\usepackage{xcolor}
\usepackage[backref]{hyperref}
\hypersetup{
    colorlinks,
    linkcolor={red!60!black},
    citecolor={green!60!black},
    urlcolor={blue!60!black}
}

\usepackage{bookmark}
\usepackage[abbrev,msc-links,backrefs]{amsrefs}
\usepackage{doi}

\renewcommand{\PrintDOI}[1]{\doi{#1}}

\usepackage[T1]{fontenc}
\usepackage{lmodern}
\usepackage{microtype}

\usepackage[english]{babel}

\usepackage{geometry}
\usepackage{enumitem}

\def\alabel{\upshape({\itshape \alph*\,})}

\let\polishlcross=\l
\def\l{\ifmmode\ell\else\polishlcross\fi}

\newcommand\qand{\quad\text{and}\quad}

\let\emptyset=\varnothing
\let\setminus=\smallsetminus

\makeatletter
\def\moverlay{\mathpalette\mov@rlay}
\def\mov@rlay#1#2{\leavevmode\vtop{
   \baselineskip\z@skip \lineskiplimit-\maxdimen
   \ialign{\hfil$\m@th#1##$\hfil\cr#2\crcr}}}
\newcommand{\charfusion}[3][\mathord]{
    #1{\ifx#1\mathop\vphantom{#2}\fi
        \mathpalette\mov@rlay{#2\cr#3}
      }
    \ifx#1\mathop\expandafter\displaylimits\fi}
\makeatother

\newcommand{\dcup}{\charfusion[\mathbin]{\cup}{\cdot}}

\usepackage{graphicx}
\usepackage{caption}
\usepackage{subcaption}
\DeclareCaptionSubType*[arabic]{figure} 
\let\eps=\varepsilon
\let\theta=\vartheta
\let\rho=\varrho

\DeclareMathOperator{\Bin}{Bin}

\newcommand{\pr}[1]{{\mathbb P}\left(#1\right)}

\def\iti#1{{\rm ({\it #1\,})}}

\def\NN{\mathbb N}

\def\ccL{\mathcal{L}^{(k)}}

\def\cA{{\mathcal A}}
\def\cH{{\mathcal H}}
\def\cL{{\mathcal L}}
\def\cM{{\mathcal M}}
\def\cF{{\mathcal F}}
\def\cB{{\mathcal B}}
\def\cC{{\mathcal C}}

\def\cP{{\mathcal P}}
\def\cQ{{\mathcal Q}}

\def\cK{{\mathcal K}}

\def\cV{{\mathcal V}}

\def\ccL{{\mathscr{L}}}
\def\ccM{{\mathscr{M}}}

\DeclareMathOperator{\Forb}{\cF}

\def\dL{{\ccL}}

\newtheorem{theorem}{Theorem}
\newtheorem{fact}[theorem]{Fact}

\newtheorem{definition}[theorem]{Definition}
\newtheorem{lemma}[theorem]{Lemma}
\newtheorem{proposition}[theorem]{Proposition}
\newtheorem{claim}[theorem]{Claim}

\begin{document}
\title[Minimum vertex degree conditions for loose Hamilton cycles]{Minimum vertex degree conditions for loose Hamilton cycles in $3$-uniform hypergraphs}

\author[Enno Bu\ss]{Enno Bu\ss}
\address{Fachbereich Mathematik, Universit\"at Hamburg, Hamburg, Germany}
\email{ennobuss@gmail.com}

\author[Hi\d{\^e}p H\`an]{Hi\d{\^e}p H\`an}
\address{Fachbereich Mathematik, Universit\"at Hamburg, Hamburg, Germany}
\curraddr{Instituto de Matem\'atica e Estat\'{\i}stica, Universidade de 
  S\~ao Paulo, S\~ao Paulo, Brazil}
\email{hh@ime.usp.br}

\author[Mathias Schacht]{Mathias Schacht}
\address{Fachbereich Mathematik, Universit\"at Hamburg, Hamburg, Germany}
\email{schacht@math.uni-hamburg.de}
\thanks{H.\ H\`an was supported by GIF grant no.~I-889-182.6/2005. 
	M. Schacht was supported through the Heisenberg-Programme of the
	DFG. The collaboration of the authors was supported by joint grant of the Deutschen Akademischen Austausch Dienst
	(DAAD) and the Funda\c{c}\~{a}o Coordena\c{c}\~{a}o de Aperfei\c{c}oamento de Pessoal de N\'\i vel (CAPES)}

\begin{abstract}
We investigate minimum vertex degree conditions for $3$-uniform hypergraphs which ensure the existence of loose Hamilton cycles. 
A loose Hamilton cycle  is a spanning cycle  
in which only consecutive edges intersect and these intersections consist of precisely one vertex.
 
 We prove that every
$3$-uniform $n$-vertex ($n$ even) hypergraph $\cH$ with minimum vertex degree  $\delta_1(\cH)\geq \left(\frac7{16}+o(1)\right)\binom{n}2$ contains a loose Hamilton cycle.
This bound is asymptotically best possible.
\end{abstract}


\maketitle

\section{Introduction}

We consider $k$-uniform hypergraphs $\cH=(V,E)$ with vertex sets $V=V(\cH)$ 
and edge sets $E=E(\cH)\subseteq \binom{V}k$, where $\binom{V}k$ denotes the family of all $k$-element subsets of the set~$V$.
We often  identify a hypergraph $\cH$ with its  edge set, i.e., $\cH\subseteq\binom{V}k$, and for an edge $\{v_1,\dots,v_k\}\in \cH$
we often suppress the enclosing braces and write $v_1\dots v_k\in\cH$ instead.
Given a $k$-uniform hypergraph $\cH=(V,E)$ and 
a set $S=\{v_1,\dots,v_s\}\in \binom{V}s$ 
let $\deg(S)=\deg(v_1,\dots,v_s)$ denote the number of edges of~$\cH$ containing the set $S$ and 
let~$N(S)=N(v_1,\dots,v_s)$ denote the set of those $(k-s)$-element sets $T\in\binom{V}{k-s}$ such that~$S\cup T$ forms an edge in $\cH$. 
We denote by $\delta_s(\cH)$ the minimum $s$-degree of $\cH$, i.e., the minimum of 
$\deg(S)$ over all $s$-element sets $S\subseteq V$. For $s=1$ the corresponding minimum degree $\delta_1(\cH)$ is  referred to as minimum vertex degree
whereas for $s=k-1$ we call the corresponding minimum degree $\delta_{k-1}(\cH)$ the minimum collective degree of $\cH$.

We  study sufficient minimum  degree conditions which enforce the existence of spanning, so-called Hamilton cycles. 
A $k$-uniform hypergraph $\cC$ is called an \emph{$\l$-cycle} if there is a  cyclic ordering
of the vertices  of $\cC$ such that every edge consists of~$k$ consecutive vertices, every
vertex is contained in an edge and
two consecutive edges (where the ordering of the edges is inherited by the
ordering of the vertices) intersect in exactly $\l$ vertices. 
For $\l=1$ we call the cycle  \emph{loose} whereas the cycle is called \emph{tight} if $\l=k-1$. 
Naturally, we say that a $k$-uniform, $n$-vertex hypergraph $\cH$ contains a Hamilton $\l$-cycle if
there is a subhypergraph of $\cH$ which forms an $\l$-cycle and which covers all
vertices of $\cH$.  Note that a Hamilton $\l$-cycle contains exactly $n/(k-\l)$ edges, implying that
the number of vertices of $\cH$ must be divisible by $(k-\l)$ which we indicate by $n\in(k-\l)\NN$.

Minimum collective degree conditions which ensure the existence of  tight Hamilton cycles were first studied in~\cite{KK} and in~\cites{RRS3,RRSk}. In particular,
in~\cites{RRS3,RRSk} R\"odl, Ruci\'nski, and Szemer\'edi found asymptotically sharp bounds for this problem. 
\begin{theorem}\label{thm:RRS}
For every $k\geq 3$ and $\gamma>0$ there exists an~$n_0$ such that every $k$-uniform hypergraph 
$\cH=(V,E)$ on $|V|=n\geq n_0$ vertices with $\delta_{k-1}(\cH)\geq (1/2+\gamma)n$ contains 
a tight Hamilton cycle.\qed
\end{theorem}

The corresponding question for loose cycles was first studied by K\"uhn and Osthus. In~\cite{KO3} they proved an asymptotically sharp bound on the minimum collective degree
which ensures the existence of loose Hamilton cycles in $3$-uniform hypergraphs. 
This result was generalised to higher uniformity by the last two authors~\cite{loosecyc} and independently by Keevash, K\"uhn, Osthus and Mycroft in~\cite{KKMO}.
\begin{theorem}
\label{thm:loosecyccodeg}
For all integers $k\geq 3$  and every $\gamma>0$ there exists an $n_0$ such that
every $k$-uniform hypergraph 
$\cH=(V,E)$ on $|V|=n\geq n_0$ vertices with  $n\in(k-1)\NN$ and $\delta_{k-1}(\cH)\geq (\frac{1}{2(k-1)}+\gamma)n$ 
contains a loose Hamilton cycle.\qed
\end{theorem}
Indeed, in \cite{loosecyc}  asymptotically sharp bounds  for Hamilton $\l$-cycles for all $\l<k/2$ were obtained. 
Later this result was generalised to all $0<\l<k$ by K\"uhn,  Mycroft, and Osthus~\cite{KMO}.
These results are asymptotically best possible for all $k$ and $0<\l<k$.
Hence, asymptotically, the problem of finding Hamilton $\l$-cycles in uniform hypergraphs with large minimum \emph{collective} degree is solved.

We focus on minimum \emph{vertex} degree conditions which ensures the existence of Hamilton cycles.
For $\delta_1(\cH)$ very few results on spanning subhypergraphs are known (see e.g. \cites{matchings, RRsurvey}).
In this paper we give an asymptotically sharp bound on the minimum vertex degree in $3$-uniform hypergraphs 
which enforces the existence of loose Hamilton cycles.
\begin{theorem}[Main result]
\label{thm:main}
For all $\gamma>0$ there exists an $n_0$ such that the following holds. Suppose $\cH$ is a $3$-uniform hypergraph on $n>n_0$ with $n\in 2\NN$ and
\[\delta_1(\cH)>\left(\frac{7}{16}+\gamma\right)\binom{n}2.\]
Then $\cH$ contains a loose Hamilton cycle.
\end{theorem}
In the proof we apply the so-called absorbing technique.
In \cite{RRS3} R\"odl, Ruci\'nski, and Szemer\'edi introduced this elegant approach to tackle minimum degree problems for spanning graphs and hypergraphs.
In our case it reduces the problem of finding a loose Hamilton cycle to the problem of finding a \emph{nearly} spanning
loose path and indeed, finding such a path will be the main obstacle to Theorem~\ref{thm:main}.

As mentioned above, Theorem~\ref{thm:main}  is best possible up to the error constant $\gamma$ as seen by the following construction from~\cite{KO3}.

\begin{fact}\label{fact:lb}
For every  $n\in 2\NN$ there exists a $3$-uniform hypergraph 
$\cH_{3}=(V,E)$ on $|V|=n$ vertices with $\delta_{1}(\cH_{3})\geq \frac{7}{16}\binom{n}2-O(n)$, which does not 
contain a loose Hamilton  cycle.
\end{fact}
\begin{proof}
Consider the following $3$-uniform hypergraph $\cH_{3}=(V,E)$.
Let $A\dcup B=V$ be a partition of $V$ with $|A|=\lfloor\frac{n}{4}\rfloor-1$ and let $E$ be the set of all triplets
from~$V$  with at least one vertex in $A$. Clearly,
$\delta_{1}(\cH_{3})=\binom{|A|}2+|A|(|B|-1)=\frac{7}{16}\binom{n}2-O(n)$.
Now consider an arbitrary cycle in $\cH_{3}$. 
Note that every vertex, in particular every vertex from $A$, is contained
in at most two edges of this cycle. Moreover,  every
edge of the cycle must intersect $A$. Consequently, the cycle contains at most
$2|A|<n/2$ edges and, hence, cannot be a Hamilton cycle.
\end{proof}

We note that the construction $\cH_{3}$ in Fact~\ref{fact:lb} satisfies $\delta_2(\cH_{3})\geq n/4-1$ and
indeed,  the same construction proves that the minimum collective degree condition given 
in Theorem~\ref{thm:loosecyccodeg} is asymptotically best possible for the case $k=3$.

This leads to the following conjecture for minimum vertex degree conditions enforcing loose Hamilton cycles in $k$-uniform hypergraphs.
Let $k\geq 3$ and let $\cH_{k}=(V,E)$ be the $k$-uniform, $n$-vertex hypergraph on  $V=A\dcup B$ with $|A|=\frac{n}{2(k-1)}-1$. Let  $E$ consist of all 
$k$-sets intersecting $A$ in at least one vertex. Then 
$\cH_{k}$ does not contain a loose Hamilton cycle
and we believe that any $k$-uniform, $n$-vertex hypergraph $\cH$ which has minimum vertex degree $\delta_{1}(\cH)\geq \delta_{1}(\cH_{k})+o(n^2)$
contains a loose Hamilton cycle. Indeed, Theorem~\ref{thm:main} verifies this for the case $k=3$.

\section{Proof of the main result}
The proof of Theorem~\ref{thm:main}
will be given in Section~\ref{sec:prooftheorem}.  
It  uses
several auxiliary lemmas which we introduce in Section~\ref{sec:auxlemmas}.
We start with an outline of the proof.

\subsection{Outline of the proof}\label{sec:outline}
We will build a loose Hamilton cycle by connecting loose paths. 
Formally, a $3$-uniform hypergraph $\cP$ is a \emph{loose path}
if there is an ordering $(v_1,\dots,v_{t})$ of its vertices 
such that every edge consists of three consecutive vertices, every vertex is contained in an edge and two consecutive
edges intersect in exactly one vertex. 
The elements $v_1$ and $v_{t}$ are called the \emph{ends} of $\cP$.
 
The Absorbing Lemma (Lemma~\ref{lem:absorb})
asserts that  every $3$-uniform hypergraph $\cH=(V,E)$ with 
sufficiently large minimum vertex degree 
contains a  so-called \emph{absorbing} loose path~$\cP$, 
which has the following property: For every set $U\subset V\setminus V(\cP)$ with 
$|U|\in 2\NN$ and $|U|\leq \beta n$ (for some appropriate $0<\beta<\gamma$)
there exists a loose path $\cQ$ with the same ends as $\cP$, which covers precisely 
the vertices $V(\cP)\cup U$. 

The Absorbing Lemma reduces the problem of finding a loose Hamilton cycle 
to the simpler problem of finding an almost spanning loose cycle, which contains 
the absorbing path $\cP$ and covers at least $(1-\beta)n$ of the vertices.
We approach this simpler problem as follows. Let $\cH'$ be the induced subhypergraph $\cH$,
which we obtain after removing the vertices of the absorbing path $\cP$ guaranteed by the Absorbing Lemma.
We remove from~$\cH'$ a 
``small'' set~$R$ of  vertices, called the \emph{reservoir} (see Lemma~\ref{lem:reservoir}),
which has the property that many loose paths can be connected to one loose cycle by using the vertices of $R$ only.

Let $\cH''$ be the remaining hypergraph after removing the vertices from $R$. 
We will choose~$\cP$ and $R$ small enough, so that $\delta_1(\cH'')\geq (\frac{7}{16}+o(1))|\binom{V(\cH'')|}{2}$.
The third auxiliary lemma, the Path-tiling Lemma (Lemma~\ref{lem:pathtiling}), asserts that all but $o(n)$ vertices of
$\cH''$ can be covered by a family of pairwise disjoint loose paths and, moreover, the number of those paths will
be constant (independent of $n$). Consequently, we can connect those paths and~$\cP$ to form a loose cycle by using 
exclusively vertices from $R$. This way we obtain a loose cycle in $\cH$, which covers all but the $o(n)$
left-over vertices from $\cH''$ and some left-over vertices from $R$. We will ensure that the number of those yet 
uncovered vertices will be smaller than $\beta n$ and, hence, we can appeal to the absorption property of $\cP$ and 
obtain a Hamilton  cycle.

As indicated earlier, among the auxiliary lemmas mentioned above the Path-tiling Lemma
is the only one for which the full strength 
of the condition $(\frac{7}{16}+o(1))\binom{n}{2}$ is required and indeed, we consider Lemma~\ref{lem:pathtiling} to be the main obstacle to proving Theorem~\ref{thm:main}.
For the other lemmas we do not attempt to optimise 
the constants.

\subsection{Auxiliary lemmas}
\label{sec:auxlemmas}

In this section we introduce the technical lemmas needed for the proof of the main theorem. 

We start with the connecting lemma which is used to connect several ``short'' loose paths to a long one.
Let $\cH$ be a $3$-uniform hypergraph and $(a_i,b_i)_{i\in[k]}$ a set consisting of~$k$ mutually disjoint pairs of vertices.
 We say that a set of triples $(x_i,y_i,z_i)_{i\in[k]}$  \emph{connects} $(a_i,b_i)_{i\in[k]}$ if 
\begin{itemize}
\item $\big|\bigcup_{i\in[k]}\{a_i,b_i,x_i,y_i,z_i\}\big|=5k$, i.e.\ the pairs and triples are all disjoint,
\item  for all $i\in[k]$ we have $\{a_i,x_i,y_i\}, \{y_i,z_i,b_i\}\in \cH$.
\end{itemize}
Suppose that  $a$ and $b$ are ends of two disjoint loose paths not intersecting  $\{x,y,z\}$ and suppose that $(x,y,z)$ connects $(a,b)$. 
Then this connection would join the two paths to one loose path.
The following lemma states that several paths can be connected, provided the minimum vertex degree is sufficiently large.

\begin{lemma}[Connecting lemma]
\label{lem:connecting}
Let  $\gamma>0$, let $m\geq 1$ be an integer, and
let $\cH=(V,E)$ be a $3$-uniform hypergraph on~$n$ vertices with 
$\delta_1(\cH)\geq \left(\frac14+\gamma\right)\binom{n}2$ and $n\geq \gamma m/12$.

For every set $(a_i,b_i)_{i\in [m]}$ of mutually disjoint pairs of distinct vertices,
there exists a set of triples $(x_i,y_i,z_i)_{i\in[m]}$  connecting $(a_i,b_i)_{i\in[m]}$.
\end{lemma}

\begin{proof}
We will find the triples $(x_i,y_i,z_i)$ to connect $a_i$ with $b_i$ for $i\in[m]$
inductively as follows.
Suppose, for some $j<k$ the triples $(x_i,y_i,z_i)$ with $i<j$ are constructed so far 
and for $(a,b)=(a_j,b_j)$ we want to find a triple $(x,y,z)$ to connect $a$ and $b$. 
Let
\[
	U=V\setminus \bigg(\bigcup_{i=1}^m\{a_i,b_i\}\cup\bigcup_{i=1}^{j-1} \{x_i,y_i,z_i\}\bigg)
\] 
and for a vertex $u\in V$ let  $L_u=(V\setminus\{u\},E_u)$ be the \emph{link graph} of $v$ defined by
\[ 
	E_u=\{vw\colon uvw\in E(\cH)\}\,.
\]
We consider $L_a[U]$ and $L_b[U]$, the subgraphs of $L_a$ and $L_b$	induced on $U$. Owing to the minimum degree condition of $\cH$ and to the assumption $m\leq \gamma n/12$, we have
\begin{equation}\label{eq:lem:conn}
	e(L_a[U])\geq \left(\frac{1}{4}+\gamma\right)\binom{n}{2}-5m(n-1)\geq \left(\frac{1}{4}+\frac{\gamma}{6}\right)\binom{n}{2}
\end{equation}
and the same lower bound also holds for $e(L_b[U]))$. Note that any pair of edges $xy\in L_a[U]$ and $yz\in L_b[U]$
with $x\neq z$ leads to a connecting triple $(x,y,z)$ for $(a,b)$. Thus, if no connecting triple exists, then
for every vertex $u\in U$ one of the following must hold: either $u$ is isolated in $L_a[U]$ or $L_b[U]$
or it is adjacent to exactly one vertex $w$ in both graphs $L_a[U]$ and $L_b[U]$. In other words, any vertex not isolated in 
$L_a[U]$ has at most one neighbour in $L_b[U]$.
Let $I_a$ be the set of isolated vertices in $L_a[U]$. Since $e(L_a[U])>\frac{1}{4}\binom{n}{2}$ we have $|I_a|<n/2$. Consequently,
\begin{align*}
e(L_b[U])&\leq \binom{|I_a|}{2}+|\{uw\in E(L_b[U])\colon u\in U\setminus I_a\}| \\
	&\leq \binom{|I_a|}{2}+(|U|-|I_a|)
	<\binom{\lfloor n/2\rfloor}{2}+n\,.
\end{align*}
Using $\gamma\leq 3/4$ and $n\geq \gamma/12$ we see that this upper bound 
violates the lower bound on~$e(L_b[U])$ from~\eqref{eq:lem:conn}. 
\end{proof}

When connecting several paths to a long one we want to make sure that the vertices used for the connection  
all come from a small set, called reservoir, which is disjoint to the paths. 
The existence of such a set is guaranteed by the following.

\begin{lemma}[Reservoir lemma]
\label{lem:reservoir}
For all $0<\gamma <1/4$  there exists an $n_0$ such that for every $3$-uniform hypergraph $\cH=(V,E)$ on $n>n_0$ vertices with minimum vertex degree
$\delta_1(\cH)\geq\left( \frac{1}{4}+\gamma\right)\binom{n}{2}$ there is a set $R$ of size at most $\gamma n$ with the following property: For every system
$(a_i,b_i)_{i\in [k]}$ consisting of  $k\leq \gamma^3n/12$ mutually disjoint pairs of vertices  from~$V$ there is a triple system connecting $(a_i,b_i)_{i\in[k]}$ which, moreover,
contains  vertices from $R$ only.
\end{lemma}

\begin{proof}
We shall show that a random set~$R$ has the required properties with positive probability.
For this proof we use some exponential tail estimates. Here we will follow a basic technique 
described in~\cite{JLR_randomgraphs}*{Section~2.6}. Alternatively Lemma~\ref{lem:reservoir}
could be deduced more directly from Janson's inequality.

For given $0<\gamma<1/4$ let $n_0$ be sufficiently large.
Let $\cH$ be as stated in the lemma and $v\in V(\cH)$. Let $L(v)$ be the link graph   defined on the vertex set $V(\cH)\setminus\{v\}$, having
the edges $e\in E(L(v))$ if $e\cup\{v\}\in \cH$. Note that $L(v)$ contains $\deg_{\cH}(v)$ edges.
Since the edge set of the omplete graph $K_n$ can be decomposed into $n-1$ edge disjoint matchings,
we can decompose the edge set of $L$ into  $i_0=i_0(v)<n$ pairwise edge disjoint matchings.
We denote these matchings by $M_1(v),\dots, M_{i_0}(v)$. 

We randomly choose a vertex set $V_{p}$ from $V$ by including each vertex $u\in V$ into~$V_p$ with probability $p=\gamma-\gamma^3$ independently. 
For every $i\in[i_0]$ let 
\[
	X_i(v)=\Big|M_i(v)\cap\binom{V_{p}}2\Big|
\] 
denote the number of edges $e\in M_i(v)$ contained in $V_{p}$. This way $X_i(v)$ is a binomially distributed random variable
with parameters $|M_i(v)|$ and $p^{2}$.
Using  the following Chernoff bounds for $t>0$ (see, e.g.,~\cite{JLR_randomgraphs}*{Theorem~2.1})
\begin{align}
\label{eq:chernoffut}
\pr{\Bin(m,\zeta)\geq m\zeta+ t}&< e^{-t^2/(2\zeta m+t/3)}\\
\label{eq:chernofflt}
\pr{\Bin(m,\zeta)\leq m\zeta- t}&< e^{-t^2/(2\zeta m)}
\end{align}
we see that 
\begin{equation}
\label{eq:reservoirsize}
\gamma \frac{n}2\leq|V_p|\leq p n+(3 n\ln 20)^{1/2}\leq\gamma n-2k
\end{equation} with probability at least $9/10$. 

Further, using \eqref{eq:chernofflt} and $|M_i(v)|\leq n/2$ we see that with probability at most~$n^{-2}$ there exists an index  $i\in[i_0]$
such that $X_i(v)\leq |M_i(v)|p^2-(3n \ln n)^{1/2}$. 
Using $\sum_{i\in[i_0]}|M_i(v)|=\deg_{\cH}(v)$ and recalling that  $\deg_{V_{p}}(v)$ denotes the degree of $v$ in $\cH[V_p\cup\{v\}]$ we obtain that 
\begin{equation}
\label{eq:reservoirdeg}
\deg_{V_{p}}(v)=\sum_{i\in[i_0]}X_i(v)\geq p^2\deg_{\cH}(v)-n (3n \ln n)^{1/2}\,,
\end{equation}
holds with probability at least $1-n^{-2}$.

Repeating the same argument for every vertex $v\in V$ we infer from the union bound 
that~\eqref{eq:reservoirdeg} holds for all vertices $v\in V$ simultaneously with probability 
at least $1-1/n$.
Hence, with positive probability we obtain a set $R$ satisfying \eqref{eq:reservoirsize} and 
\eqref{eq:reservoirdeg} for all $v\in V$.

Let $(a_i,b_i)_{i\in[k]}$ be  given and let $S=\bigcup_{i\in[k]}\{a_i,b_i\}$.
Then we have $|R\cup S|\leq \gamma n$ and
\[
\deg_{R\cup S}(v)\geq \deg_{R}(v)\geq \left(\frac14+\gamma^2\right)\binom{\gamma n}2\geq \left(\frac14+\gamma^2\right)\binom{|R\cup S|}2
\]
for all $v\in V$.
Thus, we can appeal to the Connecting Lemma (Lemma~\ref{lem:connecting}) to obtain a triple system
which connects $(a_i,b_i)_{i\in[k]}$ and  which consists of vertices from $R$ only.
\end{proof}

Next, we introduce the Absorbing Lemma which asserts the existence of a ``short'' but powerful loose path $\cP$ which can absorb any small set $U\subset V\setminus V(\cP)$.
In the following note that $\left(\frac58\right)^2<\frac7{16}$.
\begin{lemma}[Absorbing lemma]
\label{lem:absorb}
For all $\gamma>0$ there exist $\beta>0$ and $n_0$ such that the following holds.
Let $\cH=(V,E)$ be a $3$-uniform hypergraph on $n>n_0$ vertices which satisfies $\delta_1(\cH)\geq\left(\frac58+\gamma\right)^2\binom{n}2$.
Then there is a loose path $\cP$ with $|V(\cP)|\leq \gamma^{7}n$ such that for all subsets $U\subset V\setminus V(\cP)$ of size
at most $\beta n$ and $|U|\in 2\NN$ there exists a loose path~$\cQ\subset \cH$ with $V(\cQ)=V(\cP)\cup U$ and $\cP$ and $\cQ$ have
exactly the same ends.
\end{lemma}

The principle used in the proof of Lemma~\ref{lem:absorb} goes back to R\"odl, Ruci\'nski, and Szemer\'edi. They introduced
the concept  of \emph{absorption}, which, roughly speaking, stands for a local extension of a given structure, which preserves the global structure.
In our context of loose cycle we say that a $7$-tuple $(v_1,\dots,v_7)$ \emph{absorbs} the two vertices
$x,y\in V$ if 
\begin{itemize}
\item $v_1v_2v_3,v_3v_4v_5,v_5v_6v_7\in \cH$ and 
\item $v_2xv_4$, $v_4yv_6\in\cH$ 
\end{itemize}are guaranteed. In particular, $(v_1,\dots,v_7)$  and $(v_1,v_3,v_2,x,v_4,y,v_6,v_5,v_7)$ both form loose paths which, moreover, have the same ends.
The proof of Lemma~\ref{lem:absorb} relies on the following result which states that for each pair of vertices there are many $7$-tuples absorbing this pair, provided
the minimum vertex degree of $\cH$ is sufficiently large.

\begin{proposition}
\label{prop:counting}
For every $\gamma\in(0,3/8)$ there exists an $n_0$ such that the following holds. 
Suppose $\cH=(V,E)$ is a $3$-uniform hypergraph on $n>n_0$ vertices
with $\delta_1(\cH)\geq \left(\frac58+\gamma\right)^2\binom{n}2$. Then for every pair of vertices $x,y\in V$ the number of $7$-tuples absorbing $x$ and $y$
is at least~$(\gamma n)^7/8$. 
\end{proposition}
\begin{proof}
For given $\gamma>0$ we choose  $n_0=168/\gamma^7$. First we show the following.
\begin{claim}
\label{claim:countingD}
For every pair $x, y \in V(\cH)$ of vertices there exists a set $D =D(x,y)\subset V$ of size $|D| = \gamma n$  such that one of the following holds:
\begin{itemize}
\item  $\deg(x,d)\geq\gamma n$ and $\deg(y,d) \geq\frac38 n$  for all $d \in D$ or 
\item  $\deg(y,d)\geq \gamma n$ and $\deg(x,d) \geq \frac38 n$ for all $d\in D$.
\end{itemize}
\end{claim}

\begin{proof}[Proof of Claim~\ref{claim:countingD}]
By assuming the contrary there exist two vertices $x$ and $y$ such that no set~$D=D(x,y)$ fulfills Claim~\ref{claim:countingD}.

Let $A(z) = \lbrace d\in V \colon \deg(z,d) < \gamma n \rbrace$ and  let $a = {|A(x)|}/{n}$ and $b = {|A(y)|}/{n}$. Without loss of generality we assume $a \leq b$. 
There are at most  $(a+\gamma) n$ vertices $v \in V$ satisfying $\deg(y,v) \geq \frac{3}{8}n$. 
Let $B(y)=\{v\in V\colon \deg(y,v)\geq 3n/8 \}$ and note that $|B(y)|<(a+\gamma)n$.
Hence, the number of ordered pairs $(u,v)$ such that $u\in B(y)$ and $\{u,v,y\}\in \cH$ 
is at most
\[|B(y)|(n-|A(y)|)+|A(y)|\gamma n\leq (a+\gamma)(1-b)n^2+b\gamma n^2.\]
Consequently, with $2\deg(y)$ being the number of ordered pairs $(u,v)$ such that 
$\{u,v,y\}\in\cH$ we have 
\begin{align*}
\left(\left(\frac{5}{8}\right)^2 +\frac{9\gamma}{8}\right)n^2&\leq 2\deg(y)\leq  
(1-b-a-\gamma)\frac38n^2+(a+\gamma)(1-b)n^2 +2b\gamma n^2\\
&\leq \frac{n^2}{8}(5a-3b-8ab) + \frac{(3 + 8\gamma)n^2}{8},
\end{align*}

Hence, we obtain
\begin{align*}
\left(\left(\frac{5}{8}\right)^2 +\frac{9\gamma}{8}\right)n^2&\leq 2\deg(y)\leq 
\frac{3n^2}{8}(1-b) + (a + \gamma)\left(\frac{5}{8} - b\right)n^2+ 2b\gamma n^2 \\
&\leq \frac{n^2}{8}(5a-3b-8ab) + \frac{(3 + 8\gamma)n^2}{8},
\end{align*}
where in the last inequality we use the fact that $b\leq 3/8$ which is a direct consequence of the condition on $\delta_1(\cH).$
It is easily seen that this maximum is attained by $a = b= 1/8$, for which we would obtain 
\[\deg(y)\leq \left(\left(\frac{5}{8}\right)^2 +\gamma\right)n^2,\]
a contradiction.
\end{proof}

We continue the proof of Proposition~\ref{prop:counting}.
For a given pair $x,y\in V$  we will select the tuple $v_1,\dots,v_7$ such that the edges
\begin{itemize}
\item $v_1v_2v_3,v_3v_4v_5,v_5v_6v_7\in \cH$ and 
\item $v_2xv_4$, $v_4yv_6\in\cH$ 
\end{itemize}are guaranteed. Note that $(v_1,\dots,v_7)$ forms a loose path with the ends $v_1$ and $v_7$ and 
$(v_1,v_3,v_2,x,v_4,y,v_6,v_5,v_7)$ also forms a loose path with the same ends, showing that $(v_1,\dots,v_7)$ is indeed an absorbing tuple for the pair $a,b$.
Moreover, we will show that the number of choices for each $v_i$ will give rise to the number of absorbing tuples stated in the proposition.

First, we want to choose $v_4$ and let $D(x,y)$ be a set with the properties  stated in Claim~\ref{claim:countingD}. 
Without loss of generality we may assume that $|N(y,d)|\geq \frac38n$ for all $d\in D(x,y)$. Fixing some $v_4\in D(x,y)$
We choose $v_2\in N(x,v_4)$ for which there are $\gamma n$ choices. This gives rise to to hyperegde $v_2xv_4\in \cH$.
and applying  Claim~\ref{claim:countingD} to $v_2$ and $v_4$ we obtain a set $D(v_2,v_4)$ 
with the properties stated in  Claim~\ref{claim:countingD} and we choose $v_3\in D(v_2,v_4)$.
We choose $v_1\in N(v_2,v_3) $ to obtain the edge $v_1v_2v_3\in \cH$. Note that $|N(v_2,v_3)|\geq \gamma n$.
Next, we choose  $v_5 $ from  $N(v_3,v_4)$ which has size $|N(v_3,v_4)|\geq \gamma n$.  This gives rise to the edge $v_3v_4v_5\in \cH$.
We choose  $v_6$ from the set~$N(y,v_4)$ with the additional property that $\deg(v_5,v_6) \geq \gamma n/2$. Hence, we obtain~$v_4yv_6\in\cH$  and we claim that
there are at least $\gamma n/2$ such choices.
Otherwise at least  $(|N(y,v_4)|-\gamma n/2)$ vertices $v\in V$ satisfy $\deg(v_5,v) < \gamma n/2$, hence 
\[\deg(v_5) < \frac{3\gamma}{16}n^2 + \binom{(\frac{5}8+ \frac{\gamma}2)n}{2} < \delta(\cH),\]
which is a contradiction.
Lastly we choose $v_7\in N(v_5,v_6)$ to obtain the edge $v_5v_6v_7\in\cH$ which completes the absorbing tuple $(v_1,\dots,v_7)$.

The number of choices for $v_1,\dots,v_7$ is at least $(\gamma n)^7/4$ and there are at most $\binom72n^6$ choices such that $v_i=v_j$ for some $i\neq j$.
Hence, we obtain at least $(\gamma n)^7/8$ absorbing $7$-tuples for the pair $x,y$.
\end{proof}

With  Proposition~\ref{prop:counting} and the connecting lemma (Lemma~\ref{lem:connecting}) at hand the proof of the absorbing lemma follows a scheme 
which can be found in \cites{RRS3,loosecyc}. We choose a family $\cF$ of $7$-tuples by selecting each $7$-tuples with probability $p=\gamma^7 n^{-6}/448$ independently.  
Then, it is easily shown that with non-zero probability the family $\cF$ satisfies
\begin{itemize}
\item $|\cF|\leq \gamma^7n/12$,
\item for all pairs $x,y\in V$ there are at least $p\gamma^7n^7/16$ tuples in $\cF$ which absorbs $x,y$
\item the number of intersecting pairs of $7$-tuples in $\cF$ is at most $p\gamma^7 n^7/32$
\end{itemize}
We eliminate intersecting pairs of $7$-tuples by deleting one tuple
for each such pair.
By definition each for the remaining $7$-tuples $(v_1^i,\dots,v_7^i)_{i\in[k]}$ with $k\leq \gamma^7 n/12$ forms a loose path with ends $v_1^i$ and $v_7^i$ and appealing 
to Lemma~\ref{lem:connecting} we can connect them to one loose path
which can absorb any $p\gamma^7 n^7/32=\beta$ pairs of vertices, proving the lemma.
To avoid unnecessary calculations we omit the details here. \qed

The next lemma is the  main obstacle when proving Theorem~\ref{thm:main}.
It asserts that the vertex set of a $3$-uniform hypergraph $\cH$ with minimum vertex degree $\delta_1(\cH)\geq \left(\frac{7}{16}+o(1)\right)\binom{n}2$
can be almost perfectly covered by a constant number of vertex disjoint loose paths.
\begin{lemma}[Path-tiling lemma]
\label{lem:pathtiling}
For all $\gamma>0$ and $\alpha>0$ there exist integers $p$ and~$n_0$ such that for $n>n_0$ the following holds.
Suppose $\cH$ is a $3$-uniform hypergraph on $n$ vertices with minimum vertex degree 
\[\delta_1(\cH)\geq \left(\frac{7}{16}+\gamma\right)\binom{n}2.\]
Then there is a family of $p$ disjoint loose paths in $\cH$ which covers all but at most $\alpha n$ vertices of $\cH$.
\end{lemma}

The proof of Lemma~\ref{lem:pathtiling} uses the weak regularity lemma for hypergraphs and  will be given in Section~\ref{sec:pathtiling}.

\subsection{Proof of the main theorem}
\label{sec:prooftheorem}

In this section we give the proof of the main result, Theorem~\ref{thm:main}. The proof is based on the three auxiliary lemmas introduced in Section~\ref{sec:auxlemmas}
and follows the outline given in Section~\ref{sec:outline}.

\begin{proof}[Proof of Theorem \ref{thm:main}]
For given  $\gamma>0$ we apply the Absorbing Lemma (Lemma~\ref{lem:absorb}) with~$\gamma/8$ 
to obtain $\beta>0$ and $n_{\ref{lem:absorb}}$. 
We apply 
the Reservoir Lemma (Lemma~\ref{lem:reservoir}) for   
$\gamma'=\min\{\beta/ 3,\gamma/8\}$ to obtain $n_{\ref{lem:reservoir}}$ which is $n_0$ of Lemma~\ref{lem:reservoir}. 
Finally, we apply the Path-tiling Lemma (Lemma~\ref{lem:pathtiling}) with $\gamma/2$ and
$\alpha=\beta/3$  to obtain $p$ and
$n_{\ref{lem:pathtiling}}$. The $n_0$ of Theorem~\ref{thm:main} is chosen by 
\[
	n_0=\max\{n_{\ref{lem:absorb}},2n_{\ref{lem:reservoir}},2n_{\ref{lem:pathtiling}},24(p+1)/\gamma'^{\,3}\}.
\]

Now let $n\geq n_0$, $n\in2\NN$ and let $\cH=(V,E)$ be a $3$-uniform hypergraph
on~$n$  vertices with
\[\delta_1(\cH)\geq \left(\frac7{16}+\gamma\right)\binom{n}2.\]
Let $\cP_0\subset \cH$ be the absorbing path guaranteed by Lemma~\ref{lem:absorb}. Let $a_0$ and 
$b_0$ be the ends of~$\cP_0$ and note that 
\[
|V(\cP_0)|\leq \gamma'n<\gamma n/8\,.
\]
Moreover, the path $\cP_0$ has the
absorption property, i.e.,
for all  $U\subset V\setminus V(\cP_0)$  with  $|U|\leq \beta n$ \ and  $|U|\in2\NN$  there exists 
\begin{equation}
\label{eq:absorb}
\text{a loose path } \cQ\subset \cH \text{ s.t.\ }
V(\cQ)=V(\cP_0)\cup U \text{ and }\cQ \text{ has the ends } a_0 \text{ and } b_0.
\end{equation}
Let $V'=(V\setminus V(\cP_0))\cup\{a_0,b_0\}$ and let $\cH'=\cH[V']=(V',E(\cH)\cap \binom{V'}{3})$ be the
induced subhypergraph of $\cH$ on $V'$.
Note that $\delta_1(\cH')\geq (\frac7{16}+\frac34\gamma)\binom{n}2$.

Due to Lemma~\ref{lem:reservoir} we can choose a set $R\subset V'$ of 
size at most $\gamma' |V'|\leq \gamma' n$ such that 
for every system
consisting of  at most $(\gamma')^3|V'|/12$ mutually disjoint pairs of vertices  from $V$
can be connected using vertices from $R$ only. 

Set $V''=V\setminus (V(\cP_0)\cup R)$ and let $\cH''=\cH[V'']$ be the induced
subhypergraph of $\cH$ on~$V''$.
Clearly, 
\[\delta(\cH'')\geq \left(\frac7{16}+\frac{\gamma}2\right)\binom{n}{2}\]
Consequently, Lemma~\ref{lem:pathtiling} applied to $\cH''$ 
(with $\gamma_{\ref{lem:pathtiling}}$ and $\alpha$)
yields a loose path tiling of $\cH''$
which covers all but at most $\alpha|V''|\leq \alpha n$
vertices from $V''$ and which consists of at most $p$ paths. 
We denote the set of the uncovered vertices in $V''$ by $T$. 
Further, 
let $\cP_1,\cP_2\dots,\cP_{q}$ with $q\leq p$ denote the paths of the tiling.
By applying the reservoir lemma appropriately 
we connect the loose paths $\cP_0,\cP_1,\dots,\cP_q$ to one loose cycle $\cC\subset \cH$.

Let $U=V\setminus V(\cC)$ be the set of vertices not covered by the cycle $\cC$.
Since $U\subseteq R\cup T$ we have
$|U|\leq (\alpha+\gamma_{\ref{lem:reservoir}})n\leq\beta n$. 
Moreover, since $\cC$ is a loose cycle and $n\in2\NN$ we have $|U|\in2\NN$. 
Thus, using  the absorption property of $\cP_0$ (see~\eqref{eq:absorb}) we 
can replace the subpath~$\cP_0$ in $\cC$ by a path $\cQ$ (since $\cP_0$ and $\cQ$ have the same ends) 
and since $V(\cQ)=V(\cP_0)\cup U$ the resulting cycle  is  a loose Hamilton cycle of $\cH$. 
\end{proof}

\section{Proof of the   Path-tiling Lemma}
\label{sec:pathtiling}
In this section we give the proof of  the Path-tiling Lemma, Lemma~\ref{lem:pathtiling}.
Lemma~\ref{lem:pathtiling} will be  derived from the following lemma.
Let $\cM$ be the $3$-uniform hypergraph defined on the vertex set $\{1,\dots,8\}$ with the edges $123,345,456,678\in \cM$.
We will show that the condition $\delta_1(\cH)\geq \left(\frac{7}{16}+o(1)\right)\binom{n}2$ will ensure an almost perfect 
\emph{$\cM$-tiling of $\cH$}, i.e., a family of vertex disjoint copies of $\cM$, which covers almost all vertices.

\begin{lemma}
\label{lem:Mtiling}
For all $\gamma>0$ and $\alpha>0$ there exists $n_0$ such that the following holds.
Suppose $\cH$ is a $3$-uniform hypergraph on $n>n_0$ vertices with minimum vertex degree 
\[\delta_1(\cH)\geq \left(\frac{7}{16}+\gamma\right)\binom{n}2.\]
Then there is an  $\cM$-tiling of $\cH$ which covers all but at most $\alpha n$ vertices of $\cH$.
\end{lemma}
The proof of Lemma~\ref{lem:Mtiling} requires the regularity lemma which we introduce in Section~\ref{sec:wreg}.  
Sections~\ref{sec:hfunction} and~\ref{sec:Mtiling} are devoted to the proof of Lemma~\ref{lem:Mtiling} and finally, in Section~\ref{sec:pathtilingproof}, 
we deduce Lemma~\ref{lem:pathtiling} from Lemma~\ref{lem:Mtiling} by making use of the regularity lemma.

\subsection{The weak regularity lemma  and the cluster hypergraph}
\label{sec:wreg}
In this section we introduce the  \emph{weak hypergraph regularity lemma}, a 
straightforward  extension of Szemer\'edi's regularity lemma for graphs~\cite{Sz78}.
Since we only apply the lemma to $3$-uniform hypergraphs we will restrict the introduction to this case.

Let $\cH=(V,E)$ be a $3$-uniform hypergraph and let $A_1,A_2,A_3$ be
mutually disjoint non-empty subsets of~$V$. We define $e(A_1,A_2,A_3)$  
to be the number of edges with one vertex in each 
$A_i$, $i\in[3]$, and the {density} of $\cH$ with respect to $(A_1,A_2, A_3)$ as
\[
d(A_1,A_2,A_3)=\frac{e_{\cH}(A_1,A_2,A_3)}{|A_1||A_2||A_3|}\,.
\]
We say the triple $(V_1,V_2,V_3)$ of mutually disjoint subsets
$V_1,V_2,V_3\subseteq V$
is \emph{$(\eps,d)$-regular}, for  constants $\eps>0$ and $d\geq 0$,
if
\[
|d(A_1,A_2,A_3)-d|\leq \eps
\]
for all triple of subsets $A_i\subset V_i$, $i\in[3]$,
satisfying $|A_i|\geq \eps|V_i|$. 
The triple  $(V_1,V_2,V_3)$ is called \emph{$\eps$-regular} if it is
$(\eps,d)$-regular for some $d\geq 0$.
It is immediate from the definition that an $(\eps,d)$-regular triple  $(V_1,V_2,V_3)$ is  $(\eps',d)$-regular for all $\eps'>\eps$ 
and if  $V_i'\subset V_i$ has size $|V_i'|\geq c|V_i|$, then $(V_1',V_2',V_3')$ is $(\eps/c,d)$-regular.

Next we show that regular triples can be almost perfectly covered by  copies of $\cM$ provided the sizes of the partition classes obey certain restrictions.
 First note that $\cM$ is a subhypergraph of a tight path. The latter is defined 
similarly to a loose path, i.e.\  there is an ordering $(v_1,\dots,v_{t})$ of the vertices 
such that every edge consists of three consecutive vertices, every vertex is contained in an edge and two consecutive
edges intersect in exactly two vertices. 
\begin{proposition}
\label{prop:oneM}
Suppose $\cH$ is a  $3$-uniform hypergraph on $m$ vertices  with at least $dm^3$ edges. 
Then there is a tight path in $\cH$ which covers at least $2(dm+1)$ vertices.
In particular, if $\cH$ is $3$-partite with the partition classes $V_1,V_2,V_3$ and $2dm>10$
then for each $i\in[3]$ there is a copy of $\cM$ in $\cH$ which intersects $V_i$ in exactly two vertices and the other partition classes in three vertices.
\end{proposition}

\begin{proof}
Starting from $\cH$ we remove all edges containing $u, v$ 
for each pair $u,v\in V$ of vertices  such that $0<\deg(u,v)  < 2dm$. 
We keep doing this until every pair $u,v$ satisfies $\deg(u,v) = 0$ or $\deg(u,v) \geq 2dm$ in the current hypergraph $\cH'$. 
Since less than 
\[
	\binom{m}2 \cdot 2dm < dm^3\leq e(\cH)
\] 
edges were removed during the process
we know that $\cH'$ is not empty.
Hence we can pick a maximal non-empty tight path $(v_1,v_2,\dots,v_t)$ in $\cH'$.  Since the pair $v_1,v_2$ is contained in an edge in $\cH'$ it is contained in $2dm$ edges and since the path
was chosen to be maximal all these vertices must lie in the path. Hence, the chosen tight path contains at least $2(dm+1)$ vertices. This completes the first part of the proof.

For the second part, note that there is only one way to embed a tight path into a $3$-partite $3$-uniform hypergraph once the two starting vertices are fixed. Since
 $\cM$ is a subhypergraph of the tight path on eight vertices we obtain the second part of the statement by possibly deleting up to two starting vertices.
 \end{proof}

\begin{proposition}
\label{prop:embedM}
Suppose the triple $(V_1,V_2,V_3)$ is $(\eps,d)$-regular with $d\geq2\eps$ and suppose the sizes of the partition classes 
satisfy  
\begin{equation}\label{eq:embedM}
m=|V_1|\geq|V_2|\geq |V_3|  \text{ with }  5|V_1|\leq3(|V_2|+|V_3|) 
\end{equation}
and $2\eps^2 m>7$. Then there is an $\cM$-tiling of $(V_1,V_2,V_3)$  leaving at most $3\eps m$ vertices uncovered.
\end{proposition}

\begin{proof}
Note that if we take a copy of $\cM$ intersecting $V_i$, $i\in[3]$ in exactly two vertices then this copy intersects the other partition classes in exactly three vertices.
We define \[t_i=(1-\eps)\frac18\left(3|V_j|+3|V_k|-5|V_i|\right) \quad\text{where}\quad i,j,k\in[3] \text{ are distinct.}\]
Due to our assumption all $t_i$ are non-negative and we choose $t_i$ copies of $\cM$ intersecting~$V_i$ in exactly two vertices. 
This would leave $|V_i|- (2t_i+3t_j+3t_k)=\eps |V_i|$ vertices in $V_i$ uncovered, hence at most $3\eps m$ in total.

To complete the proof we exhibit a copy of $\cM$ in all three possible types 
in the remaining hypergraph, hence showing that the choices of the copies above are indeed possible.
To this end, from the  remaining vertices of each partition class $V_i$ take a subset $U_i$, $i\in[3]$ of size~$\eps |V_i|$.  
Due to the regularity of the triple $(V_1,V_2,V_3)$ we have 
$e(U_1,U_2,U_3)\geq(d- \eps) (\eps m)^3$. Hence, by Proposition~\ref{prop:oneM} there is a copy of $\cM$ (of each type) in $(U_1,U_2,U_3)$.
\end{proof}

The connection of regular partitions and dense hypergraphs is established by regularity lemmas.
The version introduced here is a straightforward generalisation of the original regularity lemma to hypergraphs (see, e.g.,~\cites{Ch91,FR92,Steger}).
\begin{theorem}
  \label{thm:wreg}
  For all $t_0 \geq 0$ and  $\eps>0$, there exist $T_0=T_0(t_0,\eps)$ and $n_0 = n_0(t_0, \eps)$
  so  that for every $3$-uniform hypergraph $\cH=(V,E)$
  on $n \geq n_0$ vertices, 
  there exists a partition
  $V=V_0\dcup V_1\dcup\dots\dcup V_t$ such that
  \begin{itemize}
  \item[\iti{i}] $t_0\leq t \leq T_0$,
  \item[\iti{ii}] $|V_1|=|V_2|=\dots= |V_t|$ and $|V_0|\leq \eps n$, 
  \item[\iti{iii}] for all but at most $\eps\binom{t}{3}$ sets
    $\{i_1,i_2,i_3\}\in\binom{[t]}3$,
     the triple $(V_{i_1},V_{i_2},V_{i_3})$ is
  $\eps$-regular.\qed
  \end{itemize}
\end{theorem} 

A partition as given in Theorem \ref{thm:wreg} is called an \emph{ $(\eps,t)$-regular partition} of $\cH$.
For an $(\eps,t)$-regular partition of $\cH$ and $d\geq 0$ we refer to $\cQ=(V_i)_{ i \in[t]}$ as the family of \emph{ clusters} (note that 
the exceptional vertex set $V_0$ is excluded)
and define the \emph{cluster hypergraph} $\cK=\cK(\eps,d,\cQ)$  with
vertex set $[t]$ and $\{i_1,i_2, i_3\}\in\binom{[t]}{3}$ being an edge if and only if 
$(V_{i_1},V_{i_2},V_{i_3})$ is $\eps$-regular and 
$d(V_{i_1},V_{i_2},V_{i_3})\geq d$.

In the following we show that the cluster hypergraph almost inherits the minimum vertex degree of the original hypergraph. 
The proof which we give for completeness  is standard and can be found e.g. in \cite{reglemsurvey} for the case of graphs.

\begin{proposition}
\label{prop:clustermindeg}
For all $\gamma>d>\eps>0$ and all $t_0$ there exist $T_0$ and $n_0\in\NN$ such that the following holds.

If $\cH$ is a $3$-uniform hypergraph on $n>n_0$ vertices with 
$\delta_1(\cH)\geq\left(\frac{7}{16}+\gamma\right)\binom{n}2$, then there exists an $(\eps,t)$-regular partition $\cQ$ 
with $t_0<t<T_0$ such that the cluster hypergraph $\cK=\cK(\eps,d,\cQ)$ 
has minimum vertex degree
 $\delta_1(\cK)\geq \left(\frac{7}{16}+\gamma-\eps-d\right)\binom{t}2.$
\end{proposition}

\begin{proof}
Let $\gamma>d>\eps$ and $ t_0$ be given. 
We apply the regularity lemma with $\eps'=\eps^2/144$ and $t_0'=\max\{2t_0,10/\eps\}$ 
to obtain $T_{0}'$ and $n_{0}'$.
We set $T_0=T_{0}'$ and $n_0=n_{0}'$.
Let $\cH$ be a $3$-uniform hypergraph on $n>n_0$ vertices which satisfies $\delta(\cH)\geq(7/16+\gamma)\binom{n}2$.
By applying the regularity lemma we obtain an $(\eps',t')$-regular partition $V'_0\dcup V_1\dcup \dots\dcup V_{t'}$ of $V$
and let $m=|V_1|=(1-\eps')n/t'$ denote the size of the partition classes.

Let $I=\{i\in[t']\colon V_i \text{ is contained in more than } \eps\binom{t'}2/8 \text{ non }\eps'\text{-regular triples}\}$ and observe that 
$|I|< 8\eps't'/\eps$ due to the property $(iii)$ of Theorem~\ref{thm:wreg}. 
Set $V_0=V'_0\cup\bigcup_{i\in I}V_i$ and let $J=[t']\setminus I$ and $t=|J|$. We now claim that $V_0$ and $\cQ=(V_j)_{j\in J}$ is the desired partition.
Indeed, we have $T_0>t'\geq t>t'(1-8\eps'/\eps)\geq t_0$ and $|V_0|< \eps'n+8\eps'n/\eps \leq \eps n/16$.
The property $(iii)$ follows directly from  Theorem~\ref{thm:wreg}.
For a contradiction, assume now that $\deg_{\cK}(V_j)<(\frac{7}{16}+\gamma-\eps-d)\binom{t}2$ for some $j\in J$. 
Let $x_j$ denote the number of edges which intersect $V_j$ 
in exactly one vertex and each other $V_i$, $i\in J$, in at most one vertex. 
Then, the assumption yields
 \begin{align*}
 x_j&\leq |V_j|\left[\left(\frac{7}{16}+\gamma-\eps-d\right)\binom{t}2m^2+\frac{\eps}8\binom{t'}2m^2+\frac{\eps}{16}n^2+
 d\binom{t}2m^2\right]\\ &\leq|V_j|\frac{n^2}2\left(\frac{7}{16}+\gamma-\frac{\eps}2\right) 
 \end{align*}
 On the other hand, from the minimum degree of $\cH$ we obtain
 \begin{align*}
 x_j\geq |V_j|\left(\frac{7}{16}+\gamma\right)\binom{n}2-2\binom{|V_j|}2n-3\binom{|V_j|}3
 \geq |V_j|\binom{n}2\left(\frac{7}{16}+\gamma-\frac{4}{t'}\right)
 \end{align*}
a contradiction.
\end{proof}

\subsection{Fractional \texorpdfstring{$\hom(\cM)$}{hom(M)}-tiling}
\label{sec:hfunction}
To obtain a large $\cM$-tiling in the hypergraph $\cH$, we consider weighted homomorphisms from $\cM$ into the cluster hypergraph $\cK$. 
To this purpose, we define the following.

\begin{definition}
\label{def:frachomM}Let $\cL$ be a $3$-uniform hypergraph. 
A function $h\colon V(\cL) \times E(\cL) \rightarrow [0,1]$ is called a fractional $\hom(\cM)$-tiling of $\cL$ if
\begin{enumerate}[label=\alabel]
\item $h(v,e) \neq 0 \Rightarrow v \in e$,
\item \label{en:pk2} $h(v) = \sum_{e \in E(\cL)}h(v,e) \leq 1$,
\item  \label{en:pk3} for every $e \in E(\cL)$ there exists a labeling of the vertices of $e = uvw$ such that \[h(u,e) = h(v,e) \geq h(w,e) \geq \frac23 h(u,e)\]
\end{enumerate}
By $h_{\min}$ we denote the smallest non-zero value of $h(v,e)$ (and we set $h_{\min}=\infty$ if $h\equiv 0$) and
 the sum over all values is the \emph{weight} $w(h)$ of $h$
  \[w(h) = \sum_{(v,e) \in V(\cL) \times E(\cL)} h(v,e)\,.\]
\end{definition}

The allowed values of $h$ are based on the homomorphisms from $\cM$ to a single edge, hence the term $\hom(\cM)$-tiling.
Given one such homomorphism, assign each vertex in the image the number of vertices from $\cM$ mapped to it.
In fact, for any such homomorphism the preimage of one vertex has size two, while the preimages 
of the other two vertices has size three. Consequently, for any family of homomorphisms of $\cM$
into a single edge the smallest and the largest class of preimages can differ by a factor of $2/3$ at most
and this observation is the reason for condition~\ref{en:pk3} in Definition~\ref{def:frachomM}.
We also note the following.
\begin{fact}
\label{fact:homM}
There is a fractional $\hom(\cM)$-tiling $h$ of the hypergraph $\cM$ which has $h_{\min}\geq 1/3$ and weight $w(h)=8$.
\end{fact}
\begin{proof}
Let $x_1$, $x_2$, $w_1$, $y_1$, $y_2$, $w_2$, $z_1$, and $z_2$ be the vertices of 
$\cM$ and let 
\[
	x_1x_2w_1\,,\quad w_1y_1y_2\,,\quad y_1y_2w_2\,,\qand w_2z_1z_2
\]
be the edges 
of~$\cM$. 
On the edges $x_1x_2w_1$ and $z_1z_2w_2$ we assign the vertex weights
$(1,1,2/3)$, where the weight $2/3$ is assigned to~$w_1$ and~$w_2$. The vertex weights 
for edges $y_1y_2w_1$ and $y_1y_2w_2$ are $(1/2,1/2,1/3)$, where $w_1$ and $w_2$ get 
the weight $1/3$. It is easy to see that those vertex weights give rise to 
a  $\hom(\cM)$-tiling $h$ on $\cM$ with $h_{\min}=1/3$ and $w(h)=8$.
\end{proof}

The notion $\hom(\cM)$-tiling is also motivated by the following 
proposition 
which shows that such a fractional $\hom(\cM)$-tiling in a cluster hypergraph  
can be ``converted'' to an integer $\cM$-tiling in the original hypergraph.

\begin{proposition}
\label{prop:Mfrac2int}
Let $\cQ$ be an $(\eps,t)$-regular partition of a $3$-uniform, $n$-vertex hypergraph~$\cH$ with $n>21\eps^{-2}$ and let
$\cK=\cK(\eps,6\eps,\cQ)$  
be the corresponding cluster hypergraph.
Furthermore, let $h\colon V(\cK)\times E(\cK)\to [0,1]$ be a fractional $\hom(\cM)$-tiling of $\cK$ with $h_{\min}\geq 1/3$.
Then there exists an $\cM$-tiling of $\cH$ which covers all but at most $(w(h)-27t\eps) |V_1|$ vertices.
\end{proposition}
\begin{proof}
We restrict our consideration to the subhypergraph $\cK'\subset \cK$ consisting of the hyperedges 
with positive weight, i.e., $e=abc\in\cK$ with $h(a), h(b),h(c)\geq h_{\min}$. 
For each $a\in V(\cK')$ let $V_a$ be the corresponding  partition class in $\cQ$. Due to the property~\ref{en:pk2} of Definition~\ref{def:frachomM} 
we can subdivide $V_a$ (arbitrarily) into a collection
of pairwise disjoint sets  $(U_a^e)_{a\in e\in \cK}$ of size $|U_a^e|=h(a,e)|V_a|$. 
Note that every edge $e=abc\in\cK$ corresponds to the $(\eps,6\eps)$-regular triplet $(V_a,V_b,V_c)$. 
Hence we obtain from the definition of regularity  and $h_{\min}\geq 1/3$ that the triplet $(U_a^e,U_b^e,U_c^e)$
is  $(3\eps,6\eps)$-regular. From the property~\ref{en:pk3} in Definition~\ref{def:frachomM} and Proposition~\ref{prop:embedM} we obtain an $\cM$-tiling of 
$(U_a^e,U_b^e,U_c^e)$  incorporating at least $\big(h(a,e)+h(b,e)+h(c,e)-9\eps\big)|V_a|$ vertices.
Applying this to all hyperedges of $\cK'$ we obtain an $\cM$-tiling incorporating at least
\[\left(\sum_{abc=e \in \cK'} h(a,e)+h(b,e)+h(c,e)-9\eps\right)|V_a|\geq \big(w(h)-9|\cK'|\eps\big)|V_a|\]
vertices. Noting that $|\cK'|\leq 3t$ (because of $h_{\min}\geq 1/3$) and $|V_a|\geq |V_1|$ we obtain the proposition.
\end{proof}

Owing to Proposition~\ref{prop:Mfrac2int}, we are given a connection between 
fractional $\hom(\cM)$-tilings of the cluster hypergraph $\cK$ of $\cH$ and $\cM$-tilings in $\cH$. 
A vertex $i \in V(\cK)$ corresponds to a class of vertices $V_i$ in the regular partition of $\cH$. 
The total vertex weight $h(i)$ essentially translates to the proportion of vertices of $V_i$
which can be covered by the corresponding $\cM$-tilings in $\cH$.
Consequently, $w(h)$ essentially translates to the proportion of vertices covered by the corresponding 
$\cM$-tiling in~$\cH$.
This reduces our task to finding a fractional $\hom(\cM)$-tiling with weight greater than the number of vertices previously covered in $\cK$.

The following lemma (Lemma~\ref{lem:frachomextent}),
which is the main tool for the proof of Lemma~\ref{lem:Mtiling}, follows the idea discussed above.
In the proof of Lemma~\ref{lem:Mtiling} we fix a maximal $\cM$-tiling 
in the cluster hypergraph $\cK$ of the given hypergraph~$\cH$. Owing to the minimum degree
condition of $\cH$ and Proposition~\ref{prop:clustermindeg},
a typical vertex in the cluster hypergraph $\cK$ will be 
contained in at least $(7/16+o(1))\binom{|V(\cK)|}{2}$ hyperedges of $\cK$.
We will show that a typical vertex $u$ of $\cK$ which is not covered by
the maximal $\cM$-tiling of $\cK$, has the property that 
$(7/16+o(1))\cdot 64 > 28$ of the edges incident to $u$ intersect 
some pair of copies of $\cM$ from the $\cM$-tiling of~$\cK$.  
Lemma~\ref{lem:frachomextent}  asserts that two such  vertices and the pair of copies of $\cM$ can be used to obtain a 
fractional $\hom(\cM)$-tiling with a weight significantly larger than~16, the number of vertices of the two copies of $\cM$.
This lemma will come in handy in the proof of  Lemma~\ref{lem:Mtiling},
where it is used to show that one can cover a higher proportion of the vertices of $\cH$
than the proportion of vertices covered by the largest $\cM$-tiling in~$\cK$.

We consider a set of hypergraphs $\dL_{29}$ definied as follows: Every $\cL\in\dL_{29}$
consists of two (vertex disjoint) copies of $\cM$, say $\cM_1$ and $\cM_2$,
and two additional vertices~$u$ and $v$ such that all edges incident to $u$ or $v$ contain precisely one vertex from $V(\cM_1)$ and 
one vertex from $V(\cM_2)$. Moreover, $\cL$ satisfies the following properties
\begin{itemize}
\item for every $a\in V(\cM_1)$ and $b\in V(\cM_2)$ we have $uab \in E(\cL)$ iff $vab \in E(\cL)$
\item $\deg(u)=\deg(v) \geq 29$.
\end{itemize}

\begin{lemma}
\label{lem:frachomextent}
For every $\cL \in \dL_{29}$ there exists a fractional $\hom(\cM)$-tiling $h$ with 
$h_{\min}\geq 1/3$ and $w(h) \geq 16 + \frac{1}{3}$.
\end{lemma}

The following proof of Lemma~\ref{lem:frachomextent} is based on straightforward, but 
somewhat tedious case distinction.

\begin{proof}
For the proof we fix the following labeling of the vertices of the two disjoint copies of $\cM$.
Let 
\[ V(\cM_1)=\{x_1,x_2,w_1,y_1,y_2,w_2,z_1,z_2\}\!\qand\!
   E(\cM_1)=\big\{x_1x_2w_1,\, w_1y_1y_2,\, y_1y_2w_2,\, w_2z_1z_2\big\} 
\] 
be the vertices 
and edges of the first copy of $\cM$. Analogously, let 
\[
	V(\cM_2)=\{x_1',x_2',w_1',y_1',y_2',w_2',z_1',z_2'\}\!\qand\! 
	E(\cM_2)=\big\{x'_1x'_2w'_1,\,w'_1y'_1y'_2,\,y'_1y'_2w'_2,\,w'_2z'_1z'_2\big\} 
\]
be the vertices and edges of the 
other copy of $\cM$ (see Figure~\ref{pic:basic}). 
Moreover, we set $X=\{x_1,x_2\}$,
$Y=\{y_1,y_2\}$, and $Z=\{z_1,z_2\}$ and, let $X'$, $Y'$, and $Z'$ be defined analogously
for $\cM_2$. 

\begin{figure}[htp]\small
	\caption{Labels and case: $a_1b_1$, $a_2b_2\in L_1$ with 
			$\{b_1,b_2\}\in \{X',Y',Z'\}$}
	\begin{subfigure}[t]{.47\linewidth}
		\setcounter{figure}{1}
		\centering
    	\includegraphics{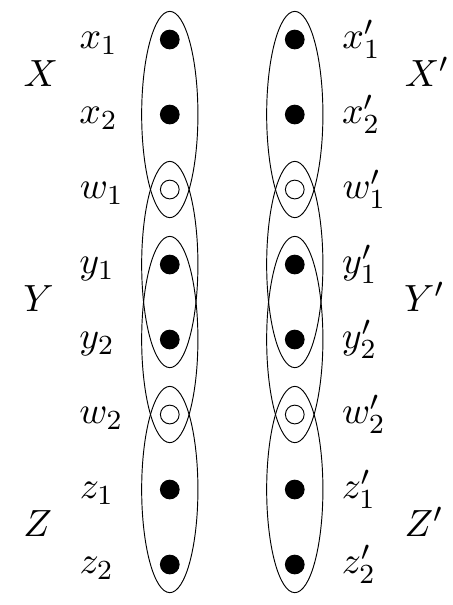}
  		\caption{Vertex labels of $\cM_1$ and $\cM_2$ in $\cL$}
  		\label{pic:basic}
	\end{subfigure}
	\hfill
	\begin{subfigure}[t]{.47\linewidth}
		\centering
    	\includegraphics{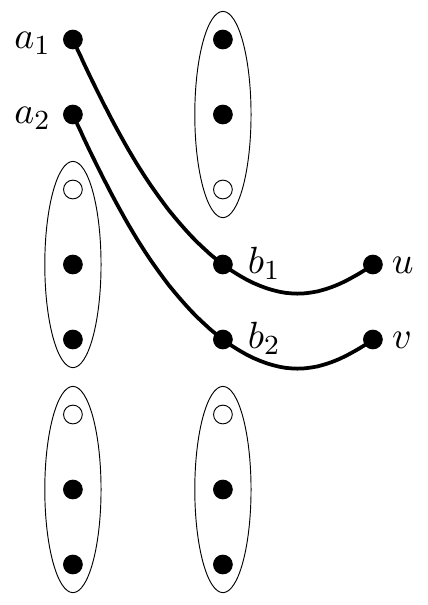}
		\caption{All edges are (a1)-edges}
  		\label{pic:case1}
	\end{subfigure}
	\setcounter{subfigure}{0}
\end{figure}
The proof of Lemma~\ref{lem:frachomextent} proceeds in two steps.
First, we show that in any possible counterexample $\cL$, the edges incident to 
$u$ and $v$ which do not contain any vertex from $\lbrace w_1,w_2,w_1',w_2' \rbrace$ form a subgraph of $K_{2,3,3}$ (see Claim~\ref{claim:enno1}). In the second step we show that 
every edge contained in this subgraph of $K_{2,3,3}$ forbids too many other edges incident to $u$ and $v$, which will yield a contradiction to the condition 
$\deg(u)=\deg(v)\geq 29$ of $\cL$ (see Claim~\ref{claim:enno2}).

We introduce the following notation to simplify later arguments. 
For a given $\cL\in\dL_{29}$ with~$\cM_1$ and $\cM_2$ being the copies of $\cM$, let~$L$ be set of the  pairs $(a,b)\in  V(\cM_1)\times V(\cM_2)$ such that $uab \in E(\cL)$.
We split $L$ into $L_1\dcup L_2$ according to 
\[
(a,b) \in
\begin{cases} 
	L_1, & \text{if $\{a,b\}\cap \{w_1,w_2, w_1',w_2' \}=\emptyset$},\\ 
	L_2, &\text{otherwise.}
\end{cases}
\]
It will be convenient to view $L_1$ and $L_2$ as bipartite graphs 
with vertex classes $V(\cM_1)$ and~$V(\cM_2)$.

We split the proof of Lemma~\ref{lem:frachomextent} into the following two claims.

\begin{claim}\label{claim:enno1}
For all $\cL \in \dL_{29}$ without a fractional $\hom(\cM)$-tiling with $h_{\min}\geq 1/3$ and 
$w(h) \geq 16 + 1/3$, we have $L_1 \subseteq K_{3,3}$, where each of the sets 
$X$, $Y$, $Z$ and $X'$, $Y'$, $Z'$ contains precisely one of the vertices of the 
$K_{3,3}$.
\end{claim}

Claim~\ref{claim:enno1} will be used in the proof of the next claim, which clearly implies 
Lemma~\ref{lem:frachomextent}.

\begin{claim}\label{claim:enno2}
Let $F=\big\{a'b'\in V(\cM_1)\times V(\cM_2)\colon\  a'\in\{w_1,w_2\} \text{ or } b'\in\{w'_1,w'_2\}\big\}$
and for every edge  $ab\in L_1$ let  $\Forb(a,b)\subseteq F$ be the set of those
$e\in F$, whose appearance in $\cL$ (i.e.\ $e\in L_2$)  implies the existence of a fractional $\hom(\cM)$-tiling~$h$ 
with $h_{\min}\geq 1/3$ and $w(h)\geq 16+1/3$.
Then there is an  injection $f\colon L_1\to F $ such that $f(a,b)\in\Forb(a,b)$ for every 
pair $ab\in L_1$.
\end{claim}

Clearly, $|F|=28$ and  $L_2\subset F$. Hence, from $|L_2|+|f(L_1)|=|L_2|+|L_1|\geq 29$ we derive that $L_2$ and $f(L_1)$ must intersect. By Claim~\ref{claim:enno2} 
this yields the desired  fractional $\hom(\cM)$-tiling and Lemma~\ref{lem:frachomextent} follows.
\end{proof}

In the proofs of Claim~\ref{claim:enno1} and Claim~\ref{claim:enno2} we will consider fractional $\hom(\cM)$-tilings $h$ which use 
vertex weights of special types. In fact, for an edge $e = a_1a_2a_3$, 
the weights $h(a_1,e)$,  $h(a_2,e)$, and $h(a_3,e)$ will be of the following forms
\begin{itemize}
\item[(a1)] $h(a_1, e) = h(a_2,e) = h(a_3,e) = 1$ 
\item[(a2)] $h(a_1, e) = h(a_2,e) = h(a_3,e) = \frac{1}{2}$
\item[(a3)] $h(a_1, e) = h(a_2,e) = h(a_3,e) = \frac{1}{3}$
\item[(b1)] $h(a_1, e) = h(a_2,e) = 1$ and  $h(a_3,e) =  \frac{2}{3}$
\item[(b2)] $h(a_1, e) = h(a_2,e) = \frac{1}{2}$ and $h(a_3,e) =  \frac{1}{3}$
\item[(b3)] $h(a_1, e) = h(a_2,e) = \frac{2}{3}$ and $h(a_3,e) =  \frac{1}{2}$
\end{itemize}
An edge that satisfies (a1) is called an (a1)-edge, etc. Note that all these 
types satisfy condition~\ref{en:pk3} of Definition~\ref{def:frachomM}.

\begin{proof}[Proof of Claim~\ref{claim:enno1}]
Given $\cL\in\dL_{29}$ satisfying the assumptions of the claim 
and with the labeling from Figure~\ref{pic:basic}.
Observe that for any $A \in \lbrace X,Y,Z \rbrace$, the hypergraph 
$\cM_1-A$ contains two disjoint edges. Similarly, 
for every $B \in \lbrace X',Y',Z' \rbrace$, $\cM_2-B$ contains two disjoint edges.

\begin{figure}[ht]\small
	\caption{Case: $ab_1$, $ab_2\in L_1$ with $\{b_1,b_2\}\in\{X',Y',Z'\}$}
	\begin{subfigure}[t]{.47\linewidth}
		\setcounter{figure}{2}
		\centering
    	\includegraphics{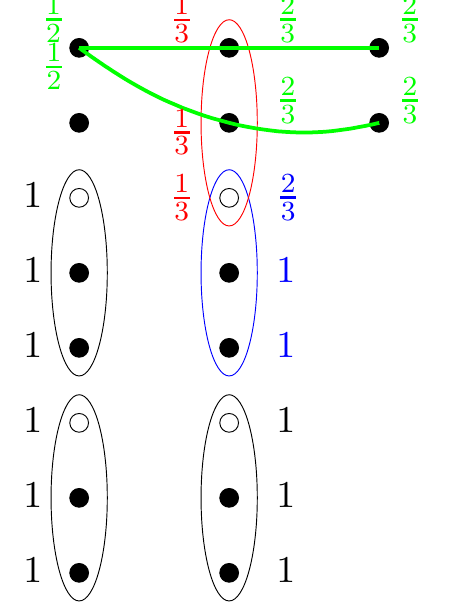}
  		\caption{(a1)-edges $w_1y_1y_2$, $w_2z_1z_2$, and $w'_2z'_1z'_2$, 
			(b3)-edges $ax_1'u$ and $ax_2'v$,  
			(b1)-edge $w_1'y_1'y_2'$, and (a3)-edge $x_1'x_2'w_1$.}
  		\label{pic:case2}
	\end{subfigure}
	\hfill
	\begin{subfigure}[t]{.47\linewidth}
		\centering
    	\includegraphics{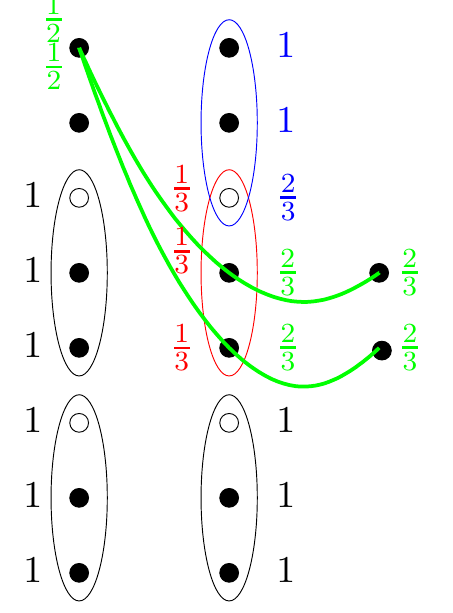}
  		\caption{(a1)-edges $w_1y_1y_2$, $w_2z_1z_2$, and $w'_2z'_1z'_2$,
				(b3)-edges $ay_1'u$ and $ay_2'v$, (b1)-edge $x_1'x_2'w_1'$, and 
				(a3)-edge $w_1'y_1'y_2'$.}
  		\label{pic:case2a}
	\end{subfigure}
	\setcounter{subfigure}{0}
\end{figure}
First we exclude the case that there is a matching $\lbrace a_1b_1, a_2b_2 \rbrace$ of size two in $L_1$ between some $
\{a_1,a_2\}=A\in \lbrace X,Y,Z \rbrace$ and some
$\{b_1,b_2\}=B \in \lbrace X',Y',Z' \rbrace$. In this case we can construct 
a fractional $\hom(\cM)$-tiling $h$ as follows: Choose two edges $ua_1b_1, va_2b_2$. 
Using these and the four disjoint edges in $(\cM_1-A)\dcup (\cM_2-B)$, we obtain six disjoint edges 
(see Figure~\ref{pic:case1}). 
Letting all these six edges be (a1)-edges, we obtain a fractional $\hom(\cM)$-tiling~$h$ with $h_{\min}=1$ and $w(h) = 18$.

We show that every $a \in A \in \lbrace X, Y, Z \rbrace$ has at most 
one neighbour in each $B \in \lbrace X', Y', Z' \rbrace$.  Assuming the contrary, let
 $a \in A \in \lbrace X, Y, Z \rbrace$ and 
$\lbrace b_1, b_2 \rbrace = B \in \lbrace X', Y', Z' \rbrace$, 
with $ab_1$, $ab_2 \in L_1$. For symmetry reasons, we only have
to consider the case $B=X'$ and~$B=Y'$. 
The case $B = Z'$ is symmetric to $B=X'$. 
In those cases, we choose $h$ as shown in  Figure~\ref{pic:case2} ($B=X'$)
and Figure~\ref{pic:case2a} ($B=Y'$) and in either case we find a fractional 
$\hom(\cM)$-tiling $h$ satisfying $h_{\min}=1/3$ and
$w(h) = 16 + 1/3$. Note that the cases $A=Y$ and $A=Z$ can be treated in the same manner since the only 
condition needed to define $h$ is that  $\cM_1-A$ contains two disjoint edges. 

To show that $L_1$ is indeed contained in a $K_{3,3}$ it remains to verify that every
 $a_1b_1$, $a_2b_2$ with 
$\{a_1, a_2 \}= A \in \{ X,Y,Z \}$, $b_1 \in B_1 \in 
\{ X',Y',Z' \}$, and $b_2 \in B_2 \in \lbrace X',Y',Z' \rbrace \setminus B_1 $ 
guarantees the existence of a fractional 
$\hom(\cM)$-tiling~$h$ with $h_{\min}\geq 1/3$ and $w(h) \geq 16 + 1/3$.
Again owing to the symmetry, the only cases we need to consider are $B_1 = X', B_2 = Y'$ (see Figure~\ref{pic:case3}) and $B_1 = X', B_2 = Z'$ (see Figure~\ref{pic:case3a}). 
In fact, the fractional 
$\hom(\cM)$-tilings~$h$ given in Figure~\ref{pic:case3} and Figure~\ref{pic:case3a}
satisfy $h_{\min}\geq 1/3$ and $w(h)=17$. Again the cases $A=Y$ and $A=Z$ can be treated in the same manner.
This concludes the proof of Claim~\ref{claim:enno1}.
\begin{figure}[thp]\small
	\caption{Case: $a_1b_1$, $a_2b_2\in L_1$ with $\{b_1,b_2\}\notin\{X',Y',Z'\}$}
	\begin{subfigure}[t]{.47\linewidth}
		\setcounter{figure}{3}
		\centering
    	\includegraphics{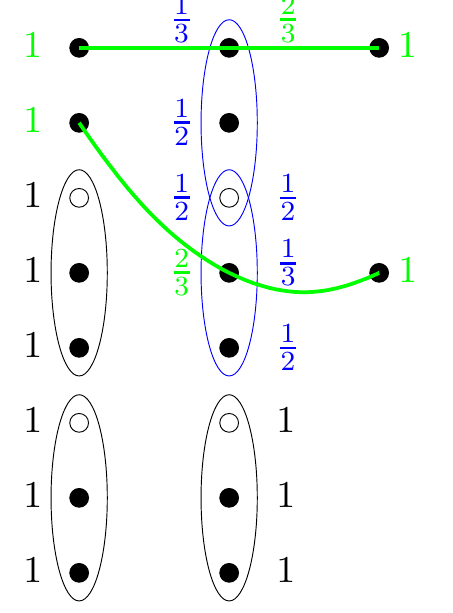}
  		\caption{(a1)-edges $w_1y_1y_2$, $w_2z_1z_2$, and $w'_2z'_1z'_2$,
			(b1)-edges $a_1x_1'u$ and $a_2y_1'v$, and
			(b2)-edges $x_1'x_2'w_1'$ and $w_1'y_1'y_2'$.}
  		\label{pic:case3}
	\end{subfigure}
	\hfill
	\begin{subfigure}[t]{.47\linewidth}
		\centering
    	\includegraphics{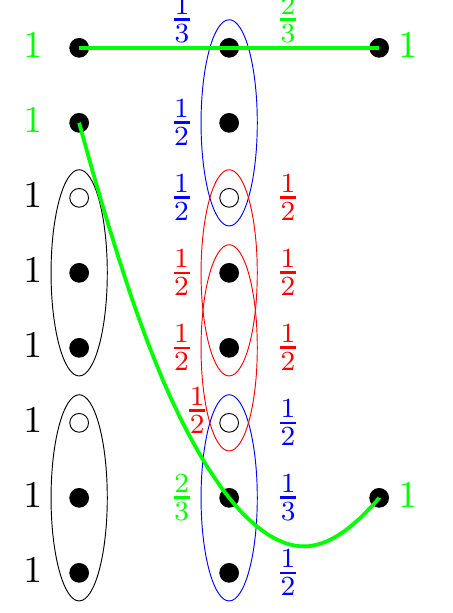}
  		\caption{(a1)-edges $w_1y_1y_2$ and $w_2z_1z_2$, 
			(b1)-edges $a_1x_1'u$ and $a_2z_1'v$,
			(b2)-edges $x_1'x_2'w_1'$ and $w_2'z_1'z_2'$, and
			(a2)-edges $w_1'y_1'y_2'$ and $y_1'y_2'w_2'$.}
  		\label{pic:case3a}
	\end{subfigure}
	\setcounter{subfigure}{0}
\end{figure}
\end{proof}

To complete the proof of  Lemma~\ref{lem:frachomextent} it is left to prove Claim~\ref{claim:enno2}.

\begin{proof}[Proof of Claim~\ref{claim:enno2}]
Before defining the injection $f\colon L_1\to F$ we collect some information about $\Forb(a,b)$ with $ab\in L_1$.
Owing to Claim~\ref{claim:enno1}, we may assume without loss of 
generality that $x_1$, $y_1$, $z_1$ and $x'_1$, $y'_1$, $z'_1$ are the vertices which span all edges of $L_1$. 
First we consider $e=y_1y_1'$.  
As shown in Figure~\ref{pic:case6} the appearance of 
$w_1y_2'\in L_1$ would give rise to a 
fractional $\hom(\cM)$-tiling~$h$ with $h_{\min}\geq 1/3$ and $w(h)=16.5$.
Consequently, we have
\[
w_1y_2'\in \Forb(y_1,y_1').
\]

For the case  $e=x_1x_1'\in L_1$, Figure~\ref{pic:case4},
shows that $x_2w'_1\in\Forb(x_1,x_1')$ and
by symmetry, it follows that 
\[
\lbrace x_2w_1',w_1x_2'\rbrace \subseteq \Forb(x_1,x_1')\,.
\]
By applying appropriate automorphisms to $\cM_1$ and $\cM_2$ we immediately obtain
information on $\Forb(x_1,z_1')$, $\Forb(z_1,x'_1)$, and $\Forb(z_1,z_1')$. Indeed, we have
\[
\lbrace x_2w_2',w_1z_2'\rbrace \subseteq \Forb(x_1,z_1')\,,\ 
\lbrace w_2x'_2,z_2w'_1\rbrace \subseteq \Forb(z_1,x_1')\,,\ 
\lbrace z_2w_2',w_2z_2'\rbrace \subseteq \Forb(z_1,z_1')\,.
\]

\begin{figure}[thp]\small
	\caption{$\Forb(y_1,y_1')$ and $\Forb(x_1,x_1')$}
	\begin{subfigure}[t]{.47\linewidth}
		\setcounter{figure}{4}
		\centering
		\includegraphics{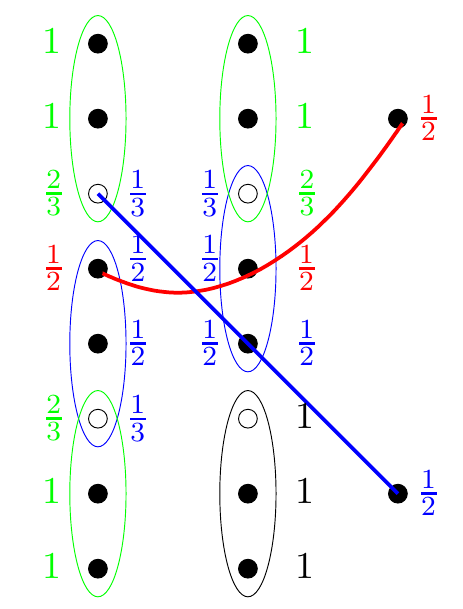}
		\caption{(a1)-edge $w'_2z'_1z'_2$, 
				(a2)-edge $y_1y_1'u$, 
				(b1)-edges $x_1x_2w_1$, $w_1z_1z_2$, and $x_1'x_2'w_1'$, and
				(b2)-edges $w_1y_2'v$, $y_1y_2w_2$, and $w_1'y_1'y_2'$.}
  		\label{pic:case6}
	\end{subfigure}
	\hfill
	\begin{subfigure}[t]{.47\linewidth}
		\centering
     	\includegraphics{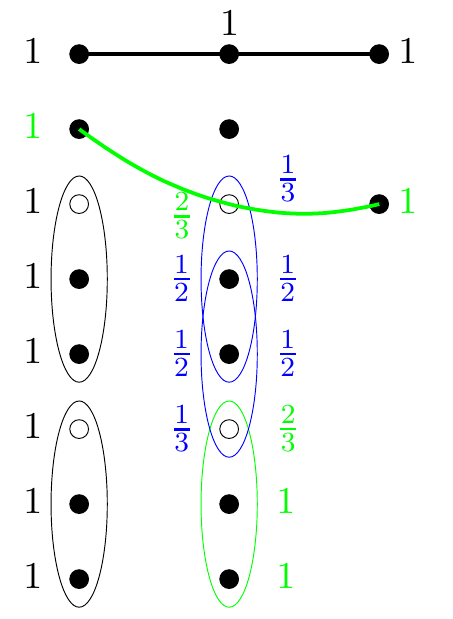}
  		\caption{(a1)-edges $x_1x'_1u$, $w_1y_1y_2$, and $w_2z_1z_2$, 
			(b1)-edges $x_2w_1'v$ and $w_2'z_1'z_2'$, and
			(b2)-edges $w_1'y_1'y_2'$ and $y_1'y_2'w_2'$.}
  		\label{pic:case4}
	\end{subfigure}
	\setcounter{subfigure}{0}
\end{figure}
Next we consider $e=y_1x_1'$. In this case Figure~\ref{pic:case4a} shows that
$y_2w_1' \in  \Forb(y_1,x_1')$. Moreover,
as shown in Figure~\ref{pic:case5} we also have
$w_1x_2'\in \Forb(y_1,x_1')$ and, consequently, we obtain
\[
	\lbrace y_2w_1', w_1x_2' \rbrace\subseteq \Forb(y_1,x_1')\,.
\]
Again applying appropriate automorphisms to $\cM_1$ and $\cM_2$ we immediately obtain
information on $\Forb(x_1,y_1')$, $\Forb(y_1,z_1')$, and $\Forb(z_1,y'_1)$. Indeed one can show
\[
\lbrace w_1y_2',x_2w_1'\rbrace \subseteq \Forb(x_1,y_1')\,,\ 
\lbrace y_2w_2',w_1z_2'\rbrace \subseteq \Forb(y_1,z_1')\,,\
\lbrace w_2y'_2,z_2w'_1\rbrace \subseteq \Forb(z_1,y_1')\,. 
\]

\begin{figure}[thp]\small
	\caption{$\Forb(y_1x_1')$}
	\begin{subfigure}[t]{.47\linewidth}
		\setcounter{figure}{5}
		\centering
    	\includegraphics{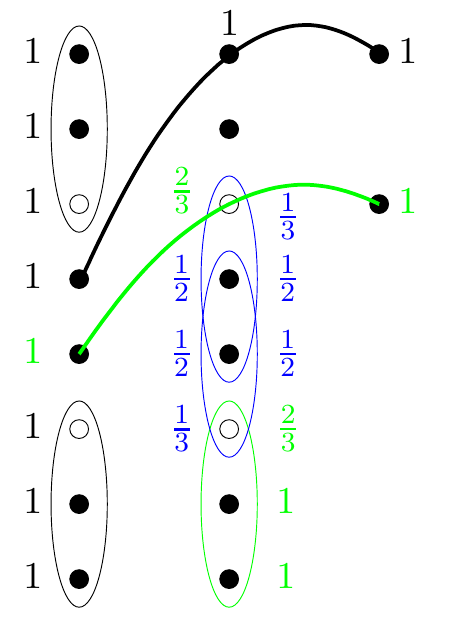}
  		\caption{(a1)-edges $y_1x_1u$, $x_1x_2w_1$, and $w_2z_1z_2$, 
			(b1)-edges $y_2w_1'v$ and $w_2'z_1'z_2'$, and
			(b2)-edges $w_1'y_1'y_2'$ and $y_1'y_2'w_2'$.}
  		\label{pic:case4a}
	\end{subfigure}
	\hfill
	\begin{subfigure}[t]{.47\linewidth}
		\centering
		\includegraphics{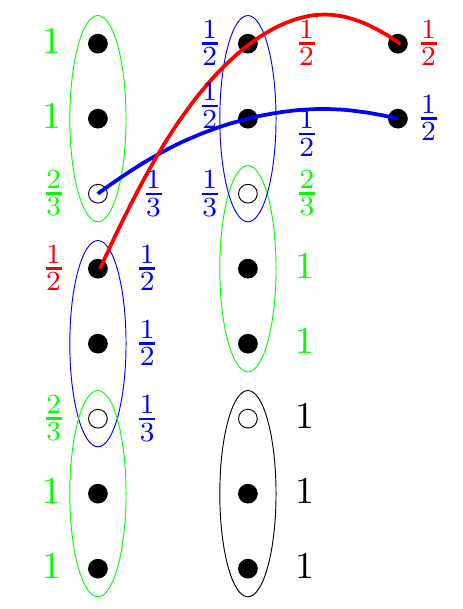}
  			\caption{(a1)-edge $w'_2z'_1z'_2$, 
					(a2)-edge $y_1x_1'u$, 
					(b1)-edges $x_1x_2w_1$, $w_2z_1z_2$, and $w_1'y_1'y_2'$, and
					(b2)-edges $w_1x_2'v$, $y_1y_2w_2$, and $x_1'x_2'w_1'$.}
  		\label{pic:case5}
	\end{subfigure}
	\setcounter{subfigure}{0}
\end{figure}

Finally, we define an injection $f\colon L_1\to F\supseteq L_2$ such that $f(a,b)\in\Forb(a,b)$ for every 
pair~$ab\in L_1$.
Recall that due to Claim~\ref{claim:enno1} we have 
$L_1\subseteq \lbrace x_1,y_1,z_1 \rbrace \times \lbrace x_1',y_1',z_1' \rbrace$
and it follows from the discussion above that we can fix $f$ as follows
\begin{align*}
f(x_1,x_1') &= w_1x_2'\,,
& f(x_1,y_1') &= x_2w_1'\,,
& f(x_1,z_1') &= x_2w_2'\,,\\
f(y_1,x_1') &= y_2w_1'\,,
& f(y_1,y_1') &= w_1y_2'\,,
& f(y_1,z_1') &= w_1z_2'\,,\\
f(z_1,x_1') &= w_2x_2'\,,
& f(z_1,y_1') &= w_2y_2'\,,
& f(z_1,z_1') &= z_2w_2'\,.
\end{align*}
Consequently, $|L_1|\leq |F|-|L_2|$ and Claim~\ref{claim:enno2} follows from 
$|L_2|+|L_1|\leq |F|\leq 28$.
\end{proof}

\subsection{Proof of the \texorpdfstring{$\cM$}{M}-tiling Lemma}
\label{sec:Mtiling}
In this section we prove Lemma~\ref{lem:Mtiling}.
Let $\cH$ be a $3$-uniform hypergraph on $n$ vertices. We say $\cH$ has a \emph{ $\beta$-deficient} $\cM$-tiling if there exists a family of pairwise disjoint copies of $\cM$ in $\cH$
leaving at most $\beta n$ vertices uncovered.
\begin{proposition} 
\label{prop:continuous}
For all $1/2>d>0$ and all $\beta, \delta>0$ the following holds.  
Suppose there exists an $n_0$ such that  every $3$-uniform hypergraph $\cH$ on $n>n_0$ vertices with minimum vertex degree $\delta_1(\cH)\geq d\binom{n}2$ has a 
$\beta$-deficient $\cM$-tiling. Then  every $3$-uniform hypergraph $\cH'$ on $n'>n_0$ vertices with $\delta_1(\cH')\geq (d-\delta)\binom{n'}2$
has a $(\beta+25\sqrt{\delta})$-deficient $\cM$-tiling.
\end{proposition}
\begin{proof}
Given a  $3$-uniform hypergraph $\cH'$ on $n'>n_0$ vertices with $\delta_1(\cH')\geq (d-\delta)\binom{n'}2$. By adding 
a set~$A$ of $3\sqrt{\delta}n'$ new vertices to $\cH'$ and adding all triplets to~$\cH'$ which intersect~$A$ we obtain a new hyperpgraph $\cH$ on $n=n'+|A|$
vertices which satisfies $\delta_1(\cH)\geq d\binom{n}2$. Consequently, $\cH$ has a $\beta$-deficient $\cM$-tiling and by removing the $\cM$-copies intersecting $A$, we obtain 
a $(\beta+25\sqrt{\delta})$-deficient $\cM$-tiling of $\cH'$.
\end{proof}

\begin{proof}[Proof of Lemma~\ref{lem:Mtiling}]
Let $\gamma>0$ be given and we assume that there is an $\alpha>0$ such that  for all $n_0'$ there is  a $3$-uniform hypergraph 
$\cH$ on $n>n_0'$ vertices which satisfies $\delta_1(\cH)\geq (\frac{7}{16}+\gamma)\binom{n}2$ but which does not contain an $\alpha$-deficient $\cM$-tiling.
Let $\alpha_0$ be the supremum  of all such $\alpha$ and note that $\alpha_0$ is bounded away from one due to Proposition~\ref{prop:oneM}.

We choose  $\eps=(\gamma\alpha_0/2^{100})^2$. By definition of  $\alpha_0$,
there is an $n_0$ such that all $3$-uniform hypergraphs $\cH$ on $n>n_0$ vertices 
satisfying $\delta_1(\cH)\geq (\frac{7}{16}+\gamma)n$ have an $(\alpha_0+\eps)$-deficient $\cM$-tiling. Hence, by 
Proposition~\ref{prop:continuous}
all $3$-uniform hypergraphs $\cH$ on $n>n_0$ vertices 
satisfying $\delta_1(\cH)\geq (\frac{7}{16}+\gamma-\eps)\binom{n}2$ have an $(\alpha_0+\eps+25\sqrt{\eps})$-deficient $\cM$-tiling. 
We will show that there exists an  $n_1$ (to be chosen) such that all $3$-uniform hypergraphs $\cH$ on $n>n_1$ vertices 
satisfying $\delta_1(\cH)\geq (\frac{7}{16}+\gamma)\binom{n}2$ have an $(\alpha_0-\eps)$-deficient $\cM$-tiling, contradicting the definition of $\alpha_0$.

We apply Proposition~\ref{prop:clustermindeg}  with the constants $\gamma$, 
$\eps/12$, $d=\eps/2$ and $t_{\ref{prop:clustermindeg}}=\max\{n_0,(\eps/12)^{-3}\}$ 
to obtain an $n_{\ref{prop:clustermindeg}}$ 
and $T_{\ref{prop:clustermindeg}}$.
Let  $n_1\geq\max\{n_{\ref{prop:clustermindeg}},n_0\}$ be sufficiently large and let $\cH$ be an arbitrary $3$-uniform hypergraph on $n>n_1$ 
vertices which satisfies   $\delta_1(\cH)\geq (\frac{7}{16}+\gamma)\binom{n}2$ but which does not contain
an $\alpha_0$-deficient $\cM$-tiling. We apply Proposition~\ref{prop:clustermindeg} to $\cH$ with the constants chosen above 
and obtain a cluster hypergraph $\cK=\cK(\eps/12,\eps/2,\cQ)$ on $t>t_{\ref{prop:clustermindeg}}$ vertices which satisfies
$\delta_1(\cK)\geq (\frac{7}{16}+\gamma-\eps)\binom{t}2$. 
Taking $\ccM$ to be the largest $\cM$-tiling in $\cK$ we know
by the definition of $\alpha_0$ and by
Proposition~\ref{prop:continuous} that $\ccM$ is
an $\alpha_1$-deficient $\cM$-tiling of $\cK$, for some $\alpha_1\leq\alpha_0+26\sqrt{\eps}$. 

We claim that $\ccM$ is not $(\alpha_0/2)$-deficient and for a contradiction, assume the contrary. 
Then, from Fact~\ref{fact:homM}, we know that for each $\cM_j\in\ccM$ 
there is a fractional $\hom(\cM)$-tiling~$h^j$ of $\cM_j$ with $h_{\min}^j\geq 1/3$ and weight $w(h^j)=8$.
Hence, the union of all these fractional $\hom(\cM)$-tiling gives rise to a fractional $\hom(\cM)$-tiling $h$ of $\cK$ with $h_{\min}\geq1/3$ and weight
\[w(h)\geq 8|\ccM|\geq t(1-\alpha_0/2).\]

By applying Proposition~\ref{prop:Mfrac2int} to the fractional $\hom(\cM)$-tiling $h$ (and 
recalling  that the vertex classes $V_1,\dots,V_t$ of the 
regular partition has the same size, which was at least $(1-\eps/12)n/t$)
 we obtain an $\cM$-tiling of $\cH$ which
covers at least 
\[\big(w(h)-3t\eps\big)\left(1-\frac{\eps}{12}\right)\frac{n}{t}
	\geq (1-\alpha_0+\eps)n \]
vertices of $\cH$.
This, however, yields a $(\alpha_0-\eps)$ deficient  $\cM$-tiling of $\cH$ contradicting the choice of $\cH$.
Hence, $\ccM$ is not $(\alpha_0/2)$-deficient from which we conclude that $X$, the set of vertices in $\cK$ not covered by~$\ccM$, has size 
\begin{equation}\label{eq:MX}
	|X|\geq \frac{\alpha_0 t}{2}\,.
\end{equation} 

For a pair $\cM_i\cM_j\in \binom{\ccM}2$  the edge $e\in \cK$ is $ij$-crossing if $|e\cap V(\cM_i)|=|e\cap V(\cM_j)|=1$. 
\begin{claim}
\label{claim:Cbig}
Let $\cC$ be the set of all triples $xij$ such that $x\in X$, $\cM_i\cM_j\in\binom{\ccM}2$ and 
there are at least $29$ $ij$-crossing edges containing $x$. Then we have
$|\cC|\geq \gamma \binom{t}2|X|/72$.
\end{claim}
\begin{proof}
Let $\cA$ be the set of those hyperedges in $\cK$ which are  completely contained in $X$ and  let $\cB$ be the set of all the edges with exactly two vertices in $X$.
Then it is sufficient to show that 
\begin{equation}
\label{eq:ABsmall}
|\cA|\leq \frac7{16}\binom{|X|}3\quad\text{and}\quad |\cB|\leq \frac7{2}\binom{|X|}2|\ccM|.
\end{equation} 
Indeed, assuming \eqref{eq:ABsmall} and $|\cC|\leq \gamma \binom{t}2|X|/72$ we
obtain the following contradiction
\begin{align*}
\sum_{x\in X}\deg(x)&\leq 3|\cA|+2|\cB|+28\left(|X|\binom{|\ccM|}2-|\cC|\right)+64|\cC|+\binom82|\ccM||X|\\
&\leq |X|\left[\frac7{16}\binom{|X|}2+\frac7{2}|X||\ccM|
		+28\binom{|\ccM|}2+\frac{36}{72}\gamma \binom{t}2 +\binom82|\ccM|\right]\\
&\leq|X| \left[\left(\frac{7}{16}+\frac{\gamma}2\right)\binom{t}2 +\binom82|\ccM|\right]\\
&<|X|\cdot\delta_1(\cK)
\end{align*}
where in the third inequality we used $\binom{t}2=\binom{|X|}2+8|X||\ccM|+\binom{8|\ccM|}2$.

Note that the first part of \eqref{eq:ABsmall} trivially holds  since in the opposite case, using the
first part of Proposition~\ref{prop:oneM} 
we obtain a tight path in $X$ of length at least eight. However, this path contains a copy of $\cM$ as a subhypergraph which yields a contradiction to the maximality of $\ccM$.

To complete the proof let us assume  $|\cB|> \frac72\binom{|X|}2|\ccM|$ 
from which
we deduce that there is an $\cM'\in\ccM$ such that $V(\cM')$ intersects at least $\frac72\binom{|X|}2$ edges from $\cB$.
From $V(\cM')$ we remove the vertices which are contained in less than $13|X|$ edges from $\cB$. Note that there are at least four vertices, say 
$v_1,\dots,v_4$, and at least $(3+\eps)\binom{|X|}{2}$ edges from $\cB$ left which intersect these vertices. 
Hence, there exists a pair $x_1, x_2$ such that $\{x_1,x_2,v_i\}\in \cH$ for all $i=1,\dots,4$. Removing all edges intersecting $x_1,x_2$ we still have at least
$(3+\eps/2)\binom{|X|}{2}$ edges intersecting $v_1,\dots,v_4$ and we can find another pair $x_3, x_4$ disjoint from $x_1, x_2$ with $\{x_3,x_4,v_i\}\in\cH$ for $i=1,\dots,4$.
For each $v_i$ we can find another edge from $\cB$ containing~$v_i$, 
keeping them all mutually disjoint and also disjoint from $\{x_1,x_2\}$ and $\{x_3,x_4\}$. This is possible 
since each $v_i$ is contained in more than $13|X|$ edges from $\cB$. This, however, yields two copies of $\cM$ which contradicts the fact that $\ccM$ was a largest possible
$\cM$-tiling.
\end{proof}
The set $X$ will be used to show that there is an $\cL\in\ccL_{29}$ such that $\cK$ contains many copies of $\cL$.
\begin{claim}
\label{claim:Lcopies}
There is an element $\cL\in\ccL_{29}$ and a family $\ccL$ of vertex disjoint copies of $\cL$ in the cluster hypergraph 
$\cK(\eps/12,\eps/2,\cQ)$ such that $|\ccL|\geq\gamma \alpha_0 t/2^{75}$.
\end{claim}

\begin{proof}
We consider the $3$-uniform hypergraph $\cC$ (as given from Claim~\ref{claim:Cbig}) on the vertex set $X\cup\ccM$.
Note that for fixed $ij$ a vertex $x$ is contained in at most $64$ $ij$-crossing edges, thus there are at most $2^{64}$ different hypergraphs  
with the property that $x$ is contained in at least 29 edges which are $ij$-crossing. 
We colour each edge $xij$ by one of the $2^{64}$ colours, 
depending on the $3$-partite hypergraph induced on $x$, $\cM_i, and \cM_j$. 
On the one hand, we observe that a monochromatic tight path consisting of the two edges $xij, x'ij\in \cC$  corresponds to a copy of~$\cL$. 
On the other hand, Claim~\ref{claim:Cbig}
implies that there is a colour such that  at least  
\[
	\frac{|\cC|}{2^{64}}
	\geq 
	\frac{\gamma\binom{t}{2}|X|}{72\cdot 2^{64} }
	\overset{\eqref{eq:MX}}{\geq} 
	\frac{\alpha_0\gamma t^3}{2^{73}}
\] edges in $\cC$ are coloured by it.
Hence, by Proposition~\ref{prop:oneM} there is a tight path with $\alpha_0\gamma t/2^{72}$ vertices 
using edges of this colour only.
Note that in this tight path every three consecutive vertices contain one vertex from $X$ and the other two vertices are from $\ccM$.
Thus, this path gives rise to at least $\alpha_0\gamma t/2^{75}$ pairwise vertex disjoint tight paths on four vertices such that the ends are vertices from $X$.
\end{proof}

For any $\cL_i\in\ccL$ we know from Lemma~\ref{lem:frachomextent}  that there is a fractional $\hom(\cM)$-tiling $h^i$ of $\cL_i$
with $h^i_{\min}\geq 1/3$ and  weight $w(h^i)\geq16+1/3$. Furthermore, for every $\cM_j\in \ccM$ which is not contained in any $\cL_i\in\ccL$ 
we know from Fact~\ref{fact:homM} that there 
is a fractional $\hom(\cM)$-tiling of $h^j$ of $\cM_j$ with $h_{\min}^j\geq 1/3$ and weight $w(h^j)=8$.
Hence, the union of all these fractional $\hom(\cM)$-tiling gives rise to a fractional $\hom(\cM)$-tiling $h$ of $\cK$ with $h_{\min}\geq1/3$ and weight
\[w(h)\geq \left(16+\frac13\right)|\ccL|+8(|\ccM|-2|\ccL|)=8|\ccM|+\frac{|\ccL|}3.\]

By applying Proposition~\ref{prop:Mfrac2int} to the fractional $\hom(\cM)$-tiling $h$  (and recalling  that the vertex classes $V_1,\dots,V_t$ of the 
regular partition has the same size which was at least $(1-\eps/12)n/t$) we obtain an $\cM$-tiling of $\cH$ which
covers at least 
\[\big(w(h)-3t\eps\big)\left(1-\frac{\eps}{12}\right)\frac{n}{t}
	\geq \left(8|\ccM|+\frac{|\ccL|}{3}-3t\eps\right)\left(1-\frac{\eps}{12}\right)\frac{n}{t} \]
vertices of $\cH$. 

Since $\ccM$ was an $(\alpha_0+26\sqrt{\eps})$-deficient $\cM$-tiling of $\cK$, the tiling we obtained above is an
$(\alpha_0-\eps)$-deficient $\cM$-tiling of $\cH$ due to the choice of $\eps$. This, however, is a contradiction to the fact that 
$\cH$ does not permit an $(\alpha_0-\eps)$-deficient $\cM$-tiling.
\end{proof}

\subsection{Proof of the path-tiling lemma}
\label{sec:pathtilingproof}
In this section we prove Lemma~\ref{lem:pathtiling}. The proof will use the following proposition which has been proven in~\cite{loosecyc} (see Lemma~20) 
in an even more general form, hence we omit the proof here.
\begin{proposition}
\label{prop:pathinreg}
  For all $d$ and $\beta>0$ there exist $\eps>0$, 
  integers $p$ and $m_0$ such that for all $m>m_0$ the following holds.
  Suppose $\cV=(V_1, V_2, V_{3})$ is an $(\eps,d)$-regular triple with
  $|V_i|=3m$ for $i=1,2$ and $|V_3|=2m.$
  Then there there is a loose path tiling of $\cV$ which consists of at most $p$ pairwise vertex disjoint paths
  and which covers all but at most $\beta m$ vertices of $\cV$.
\end{proposition}
With this result at hand one can easily derive the path-tiling lemma (Lemma~\ref{lem:pathtiling}) from the $\cM$-tiling lemma (Lemma~\ref{lem:Mtiling}).

\begin{proof}[Proof of Lemma~\ref{lem:pathtiling}]
Given $\gamma>0$ and $\alpha>0$ we first apply Proposition~\ref{prop:pathinreg} with $d=\gamma/3$ and $\beta=\alpha/4$ to obtain
$\eps'>0$, $p'$, and $m_0$. Next, we apply
Lemma~\ref{lem:Mtiling} with $\gamma/2$ and $\alpha/2$ to obtain 
 $n_{\ref{lem:Mtiling}}$. 
Then we apply Proposition~\ref{prop:clustermindeg} with $\gamma$, $d$ and $\eps=\frac13\min\{d/2,\eps',\alpha/8\}$ from above and $t_0=n_{\ref{lem:Mtiling}}$  to obtain $T_0$ and
$n_{\ref{prop:clustermindeg}}$. Lastly we set  $n_0=\max\{n_{\ref{prop:clustermindeg}}, 2T_0m_0\}$ and $p=p'T_0$.

Given a $3$-uniform hypergraph $\cH$ on $n>n_0$ vertices which satisfies $\delta_1(\cH)\geq \left(\frac7{16}+\gamma\right)\binom{n}2$. By 
applying Proposition~\ref{prop:clustermindeg} with the constants chosen above we obtain an $(\eps,t)$-regular partition
$\cQ$. Furthermore, we know that the corresponding cluster hypergraph $\cK=\cK(\eps,d,\cQ)$ satisfies $\delta_1(\cK)\geq (7/16+\gamma/2)\binom{t}2$.
Hence, by Lemma~\ref{lem:Mtiling} we know that there is an $\cM$-tiling $\ccM$ of $\cK$ which covers all but at most $\alpha t/2$ vertices of $\cK$.
Note that the corresponding vertex classes in $\cH$ contain all but  at most $\alpha n/2+|V_0|$ vertices.

We want to apply Proposition~\ref{prop:pathinreg} to each copy $\cM'\in \ccM$ of $\cM$. To this end, let $\{1,\dots,8\}$ denote the vertex set of such an copy 
$\cM'$ and let $123,345,456,678$ denote the 
edges of $\cM'$. Further, for each $a\in V(\cM')$ let $V_a$ denote the corresponding partition class
in $\cH$.  We split $V_i$, $i=3,4,5,6$, into two disjoint sets $V_i^1$ and $V_i^2$ of sizes $|V_i^1|=2|V_i|/3$ and $|V_i^2|=|V_i|/3$ for $i=3,6$ and
$|V_i^1|=|V_i^2|=|V_i|/2$ for $i=4,5$. Then the tuples $(V_1,V_2,V_3^1)$, $(V_8,V_7,V_6^1)$ and $(V_3^2,V_4^1,V_5^1)$, $(V_4^2,V_5^2,V_6^2)$ all
satisfy the condition of Proposition~\ref{prop:pathinreg}, hence, there is a path tiling of these tuples consisting of at most $4p'$ paths
which covers all but at most~$12\beta n/t$ vertices of $V_1,\dots,V_8$.

Since $\ccM$ contains at most $t/8$ elements we obtain a path tiling which consists of at most~$4p't/8\leq p'T_0/2\leq p$
paths which covers all but at most $12\beta n/t\times t/8$ vertices. Consequently, the total number of vertices in $\cH$ not covered by the path tiling is at most~$3\beta n/2+\alpha n/2+|V_0|\leq \alpha n$. This completes the proof of Lemma~\ref{lem:pathtiling}.
\end{proof}

\end{document}